\documentclass[10pt]{amsart}
\newif\ifpdf
    \ifx\pdfoutput\undefined
    \pdffalse 
    \else
    \pdfoutput=1 
    \pdftrue
    \fi

    \ifpdf
    \usepackage[pdftex]{graphicx}
    \else
    \usepackage{graphicx}
    \fi
\usepackage{amsmath,amssymb,amsthm,epsfig}

\newcommand\Z{\mathbb Z}
\newcommand\R{\mathbb R}
\newcommand\N{\mathbb N}

\newcommand\G{\mathfrak G}
\newcommand\M{\mathfrak M}
\newcommand\K{\mathfrak K}
\newcommand\1{\mathtt I}

\newtheorem{theorem}{Theorem}[section]
\newtheorem{proposition}[theorem]{Proposition}
\newtheorem{lemma}[theorem]{Lemma}
\newtheorem{corollary}[theorem]{Corollary}
\theoremstyle{definition}
\newtheorem{definition}[theorem]{Definition}

\theoremstyle{example}

\newtheorem{open}{Open Problem}

\newtheorem*{main}{Theorem 1.1}

\input{epsf.tex}

\begin{document}
\ifpdf
    \DeclareGraphicsExtensions{.pdf, .jpg, .tif}
    \else
    \DeclareGraphicsExtensions{.eps, .jpg}
    \fi

\title[Mapping Schemes Realizable by Obstructed Topological Polynomials]{Mapping Schemes Realizable by\\Obstructed Topological Polynomials}
\author {Gregory A. Kelsey}

\address{Department of Mathematics\\ 1409 W. Green St.\\ Urbana, IL 61801}
\email{gkelsey2@math.uiuc.edu}
\urladdr{http://www.math.uiuc.edu/~gkelsey2/}

\thanks{The author acknowledges support from National Science Foundation
grant DMS 08-38434 ``EMSW21-MCTP: Research Experience for Graduate Students".}


\maketitle

\begin{abstract}
In 1985, Levy used a theorem of Berstein to prove that all hyperbolic topological polynomials are equivalent to complex polynomials. We prove a partial converse to the Berstein-Levy Theorem: given post-critical dynamics that are in a sense strongly non-hyperbolic, we prove the existence of topological polynomials which are not equivalent to any complex polynomial that realize these post-critical dynamics. This proof employs the theory of self-similar groups to demonstrate that a topological polynomial admits an obstruction and produces a wealth of examples of obstructed topological polynomials.
\end{abstract}

\tableofcontents

\section{Introduction}
\label{sec:Introduction}

The modern theory of complex rational maps began with the work of Fatou and Julia during World War I \cite{Jul18}, \cite{Fat19}, \cite{Fat20}. Their theory (exposited well in \cite{Bla84}) shows that for $f$ a complex rational function of degree $d \geq 2$, the Riemann sphere splits into two disjoint sets (now known as the Fatou and Julia sets of $f$) based on the dynamics of $f$. In the 1970s, Mandelbrot used computers to produce remarkable images of these sets and of the separating curves between them \cite{Man82}. Mandelbrot particularly studied the Julia sets of quadratic complex polynomials, and he characterized when these sets are connected (this yielded the famous Mandelbrot set).

The early 1980s saw an explosion of research in this area, due in part to Mandelbrot's work. Douady and Hubbard made important advances in the theory of complex polynomials, especially quadratics \cite{DH82}, \cite{DH84}. To this day, quadratic polynomials remain the best understood in this theory (see e.g. \cite{BS02}). However, the work of Bielefeld, Fisher, and Hubbard on preperiodic polynomials along with the work of Poirer on periodic polynomials have greatly improved our understanding outside the quadratic case \cite{BFH92}, \cite{Poi09}. Non-polynomial rational maps have proved more difficult to study; most results deal only with those maps with low degree and small post-critical set (see e.g. \cite{BBLPP00}).

In the 1980s, Douady and Hubbard employed a procedure now know as \emph{mating} to combine pairs of quadratic polynomials to produce quadratic rational maps. This would allow them to apply to rational functions their machinery for polynomials. Interestingly, they found that mating some pairs of quadratic polynomials does not produce rational maps, and so the question arose as to when two quadratic polynomials are `mateable'.

To answer the mateablity question, Thurston considered a family of branched covering maps from the sphere to itself that topologically resemble complex rational maps. These maps became commonly referred to as Thurston maps, and Thurston characterized when these maps are equivalent to complex rational maps by the existence or non-existence of obstructing multicurves (see Theorem \ref{thm:thurston} in this paper) \cite{DH93}. Researchers have also considered \emph{topological polynomials}, which are Thurston maps that behave like complex polynomials.

The mateability of quadratic complex polynomials has since been solved by Rees, and others have contributed to this general area \cite{Ree86} \cite{Tan92} \cite{Shi00}. However, we remain interested in Thurston's theory partially because of its implications outside of complex dynamics. In fact, the concepts in the preceding paragraph have analogues in the theory of three-manifolds. For instance, a Thurston map admitting an obstruction is analogous to a compact, oriented, irreducible three-manifold having a nonperipheral incompressible torus. For more details regarding this connection, see the survey papers of McMullen: \cite{McM91} \cite{McM94}. Thurston proved his characterization and rigidity theorem using Teichm\"{u}ller theory, so naturally links exist there as well.

Much about Thurston obstructions remains mysterious. While producing an obstruction for a specific example may not be difficult, no one has yet discovered an algorithm for determining the existence or non-existence of an obstructing multicurve in the general setting (although Pilgrim has found that if an obstruction exists, it must be of a canonical form) \cite{Pil01}. The Berstein-Levy theorem for hyperbolic topological polynomials (see Theorem \ref{thm:levy}) remains the best result for the \emph{non}-existence of an obstruction \cite {Lev85}. More recently, Kameyama and Pilgrim have established algebraic criteria for Thurston equivalence of rational maps, but a general algorithm remains elusive \cite{Kam01} \cite{Pil03}.

Topological polynomials which admit obstructions have not seen much study. Usually, researchers use Thurston's characterization to topologically construct complex polynomials. However, Ha\"{i}ssinsky and Pilgrim \cite{HP09} and Bonk and Meyer \cite{BM} have used obstructed topological polynomials to define metrics on the sphere that are not quasisymmetric to the standard sphere. Such metrics interest analysts who seek geometric criteria for quasisymmetric equivalence to the standard sphere. The motivation from this problem comes from Cannon's Conjecture.

Our main result serves as a partial converse to the Berstein-Levy Theorem:

\begin{theorem}
\label{thm:main}

Suppose that a polynomial mapping scheme satisfies one of the following conditions:

\begin{enumerate}
\item at least one (non-attractor) period of length at least two and not containing critical values, 
\item at least two (non-attractor) periods not containing critical values,
\item at least two non-attractor periods both of length at least two, or
\item at least four non-attractor periods.
\end{enumerate}

Then this scheme is realized by a topological polynomial that is not equivalent to any complex polynomial.

\end{theorem}

Fortunately, new tools from the theory of self-similar groups have proved very powerful in the study of post-critically finite complex rational maps. We use these tools to prove Theorem \ref{thm:main}.

A group of automorphisms of an infinite rooted $d$-ary tree is said to be \emph{self-similar} if the restriction of the action of any group element on the subtree below any vertex (which is isomorphic to the entire tree) is another element of the group. This element is called the \emph{restriction} of the original group element at that vertex. Equivalently: a group is self-similar if it can be generated by a finite-state automaton.

The prototypical self-similar group is the Grigorchuk group, introduced in 1980 \cite{Gri80}. This group is \emph{contracting}; that is, if we fix a vertex sufficiently far from the root of our tree, we will have that the mapping that takes a group element to its restriction at the fixed vertex is decreasing in terms of the wordlength of the group. Grigorchuk used the contracting property to prove a variety of interesting results about his group, particularly that it exhibits what were at the time new types of growth and amenability \cite{Gri84}. The study of self-similar groups grew out of the power of the techniques Grigorchuk employed.

A decade after Grigorchuk introduced his group, Fabrykowski and Gupta defined their own group with intermediate growth \cite{FG91}. Bartholdi and Grigorchuk studied the Schreier graphs of the action of this group on the levels of the tree and found that these graphs converge to a fractal set \cite{BG00}. This work helped to inspire Nekrashevych to define the limit set of a contracting self-similar group. He then related these self-similar limit sets to the fractal Julia sets of complex rational maps by defining the \emph{iterated monodromy group} ($IMG$) of a such a map. He showed that the limit set of the $IMG$ of a rational map is homeomorphic to the map's Julia set \cite{Nek05}. Earlier, Pilgrim had also considered a monodromy action by an absolute Galois group on the set of Belyi polynomials \cite{Pil00}.

Iterated monodromy groups have proved to be a rich source of interesting groups. Grigorchuk and \.{Z}uk have studied the properties of $IMG(z^2-1)$, also known as the Basilica Group \cite{GZ02a}, \cite {GZ02b}. Their work, along with a result of Bartholdi and Virag, shows that this group is an example of a new kind of amenability \cite{BV05}. Bux and P\'{e}rez have studied the properties of $IMG(z^2+i)$ and shown it to have intermediate growth like Grigorchuk's group \cite{BP06}. Additionally, the Fabrykowsi-Gupta group is, in fact, an iterated monodromy group of a cubic polynomial \cite{Nek05}.

Recently, Bartholdi and Nekrashevych used the theory of self-similar groups to tackle questions of Thurston equivalence and Thurston obstructions \cite{BN06}. Using the iterated monodromy groups of topological polynomials satisfying particular post-critical dynamics, Bartholdi and Nekrashevych defined a new self-similar group of actions of the pure mapping class group on these polynomials and proved that it is contracting. Since the restriction map on this group leaves the Thurston equivalence class invariant, the contracting property allowed them to algebraically determine the Thurston equivalence class of these topological polynomials.

Nekrashevych continued this work to provide a description of topological polynomials and their iterated monodromy groups in terms of twisted kneading automata \cite{Nek09}. These automata encode all the topological data needed to determine the Thurston equivalence class of the topological polynomial. Further, every self-similar group generated by a twisted kneading automaton is isomorphic to the iterated monodromy group of some topological polynomial. This characterization allows us to translate topological and dynamical questions about these polynomials into algebraic questions which we can answer with explicit computations in these groups.

We utilize this characterization of topological polynomials to prove Theorem \ref{thm:main} and produce many examples of obstructed topological polynomials realizing these mapping schemes. We determine whether multicurves are obstructions to a polynomial $f$ by considering the Dehn twists about the curves as elements of the self-similar group constructed from the iterated monodromy groups of polynomials with similar post-critical dynamics to $f$. If restriction map acts cyclically on these Dehn twists, then the multicurve must be an obstruction.

While this result does not provide a complete categorization of mapping schemes in terms of their realizability by obstructed topological polynomials, we do discuss some aspects of mappings schemes requiring further study in order to prove such a characterization.

The author thanks Ilya Kapovich, Kevin Pilgrim, Sarah Koch, and Volodia Nekrashevych for their helpful conversations. 

\section{Thurston equivalence of topological polynomials}
\label{sec:thurston}

In this section, we give the standard definitions and results regarding Thurston equivalence of topological polynomials. An interested reader may find a more thorough discussion in \cite{Pil03}.

\begin{definition}
\label{def:postcritical}
For $f : S^2 \to S^2$ a branched cover of the sphere and $C_f$ the set of its critical (i.e. branching) points, we define the \emph{post-critical set} of $f$ to be the forward orbits of points in $C_f$. That is: $$P_f = \cup_{\omega \in C_f}\cup_{n\geq1}f^{\circ n}(\omega),$$ where $f^{\circ n}$ is the composition of $f$ with itself $n$ times.
\end{definition}

If $f$ is post-critically finite (i.e. if $P_f$ is a finite set), we say that $f$ is a \emph{Thurston map}.

\begin{definition}
\label{def:toppoly}
A Thurston map is a \emph{topological polynomial} if there exists some $\omega \in C_f$ such that $f^{-1}(\omega) = \{\omega\}$ (we will call this point $\infty$).
\end{definition}

We say that the \emph{degree} of a topological polynomial is the number of sheets of the covering.

Two Thurston maps $f$ and $g$ are said to be Thurston equivalent (henceforth, simply \emph{equivalent}) if there exist orientation-preserving homeomorphisms $\phi_0, \phi_1: S^2 \to S^2$ with $\phi_0(P_f) = \phi_1(P_f) = P_g$ that are isotopic relative to $P_f$ such that $\phi_0f = g\phi_1$. That is, if the following diagram commutes:

$$\begin{array}{ccc}
			(S^2,P_f) & \overset{\phi_0}{\longrightarrow}    & (S^2, P_g)     \\
			f \downarrow    &  & \downarrow g    \\
			(S^2, P_f) &\overset{\phi_1}{\longrightarrow} & (S^2, P_g)
		\end{array}
$$

In Figure \ref{fig:examplepolys}, we give diagrams in the style of \cite{BN06} of two different topological polynomials. For these diagrams, we choose a basepoint $t$ near infinity and draw a loop around $\infty$ in the negative direction. By passing to a homotopic map, we may assume that the topological polynomial takes this loop to itself by a degree $d$ mapping (where $d$ is the degree of the polynomial) which fixes our basepoint $t$. We call this loop the \emph{circle at infinity}. The point $t$ has $d$ preimages: $\{t = t_0, t_1, ... , t_{d-1}\}$, all on the circle at infinity. The preimages of our circle at infinity are subpaths of the circle, starting at some $t_i$ and ending at $t_{i+1}$ (adding mod $d$).

In Figure \ref{fig:examplepolys} we include the generators $\{s_1, s_2, s_3\}$ of $\pi_1(S^2 \setminus P, t)$ and their preimages to help demonstrate the mapping. Both topological polynomials fold the horizontal line at the critical point, which is the preimage of $\omega_1$ (the post-critical point surrounded by $s_1$).

\begin{figure}[h]
\includegraphics[width=5in]{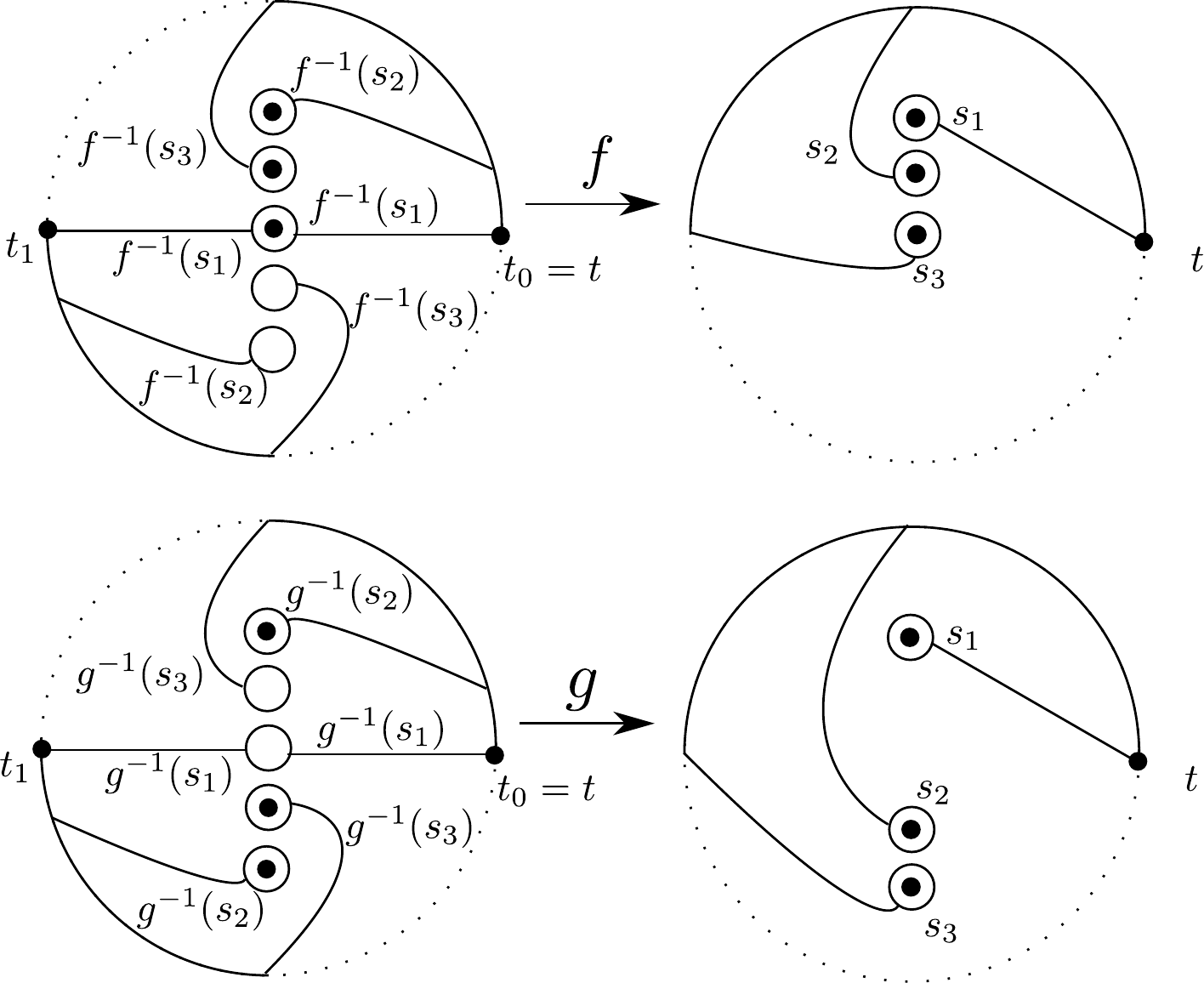}
\caption{Example topological polynomials}
\label{fig:examplepolys}
\end{figure}

To a Thurston map $f$, we associate a topological orbifold $\mathcal{O}_f$ which has underlying space $S^2$ and weight $\nu(x)$ at $x \in S^2$ equal to the least common multiple of the local degree of $f$ over all iterated preimages of $x$. The Euler characteristic $$\chi(\mathcal{O}_f) = 2 -\sum_{\omega \in P_f} (1 - \frac{1}{\nu(\omega)})$$ of this orbifold is always non-positive. If it is negative, we say that the orbifold is \emph{hyperbolic}. 

\begin{definition}
\label{def:nonperipheral}
A simple closed curve $\gamma$ on $S^2 \setminus P$ is \emph{non-peripheral} if both components of $S^2 \setminus \gamma$ contain at least two points in $P$.
\end{definition}

A \emph{multicurve} $\Gamma = \{\gamma_1, \gamma_2, ... , \gamma_n\}$ is a set of non-peripheral simple closed curves on $S^2 \setminus P_f$ that are disjoint and pairwise non-homotopic. We say that a multicurve $\Gamma$ is \emph{$f$-stable} if for all $\gamma \in \Gamma$, we have that every non-peripheral component of $f^{-1}(\gamma)$ is homotopic to some curve in $\Gamma$. For $\Gamma$ stable under $f$, there is an induced map $f_\Gamma: \R^\Gamma \to \R^\Gamma$ given by:
\[ f_\Gamma(\gamma_i) = \sum_{\delta \in f^{-1}(\gamma_i)}\frac{\delta}{\deg f|_\delta} \] By the Perron-Frobenius Theorem, there is a leading positive real eigenvalue $\lambda(f_\Gamma)$ that realizes the spectral radius of $f_\Gamma$.

We can now state Thurston's characterization and rigidity theorem:

\begin{theorem}\cite{DH93}
\label{thm:thurston}
A Thurston map $f$ with hyperbolic orbifold is equivalent to a rational function if and only if for any stable multicurve $\Gamma$, $\lambda(f_\Gamma) < 1$. In that case, the rational function is unique up to conjugation by a M\"{o}bius transformation.
\end{theorem}

\begin{definition}
\label{def:obstruction}
A stable multicurve $\Gamma$ such that $\lambda(f_\Gamma) \geq 1$ is called an \emph{obstruction}. 
\end{definition}

Unfortunately, there is no known algorithm for determining whether a Thurston map admits an obstruction.

For topological polynomials, we can restate Thurston's theorem as follows:

\begin{theorem}\cite{BFH92}
A topological polynomial is equivalent to a complex polynomial if and only if it admits no obstructions. In that case, the complex polynomial is unique up to conjugation by an affine transformation.
\end{theorem}

The obstructions to topological polynomials are better understood.

\begin{definition}
\label{def:levy}
For $f$ a topological polynomial and $\Gamma$ an obstruction that it admits, then a \emph{Levy cycle} is a set $\Gamma' = \{ \gamma_0, \gamma_1, ... , \gamma_{n-1}\} \subseteq \Gamma$ such that each $f^{-1}(\gamma_i)$ has exactly one non-peripheral component $\tilde{\gamma}_{i-1}$ homotopic to a curve in $\Gamma'$, $\tilde{\gamma}_{i-1}$ is homotopic to $\gamma_{i-1}$, and the map $f: \tilde{\gamma}_{i-1} \to \gamma_i$ has degree 1 (subtracting mod $n$).
\end{definition}

\begin{theorem}\cite{Lev85}

For $f$ a topological polynomial and $\Gamma$ an obstruction that it admits, then $\Gamma$ contains a Levy cycle.

\end{theorem}

Note that a Levy cycle need not be stable (as in the definition in \cite{BN06}). However, in this paper we will only consider stable Levy cycles, since they are easier to identify with our method.

The topological polynomial $g$ in Figure \ref{fig:examplepolys} admits a Levy cycle consisting of a single curve $\Gamma$, as shown in Figure \ref{fig:obstructions}.

\begin{figure}[h]
\includegraphics[width=5in]{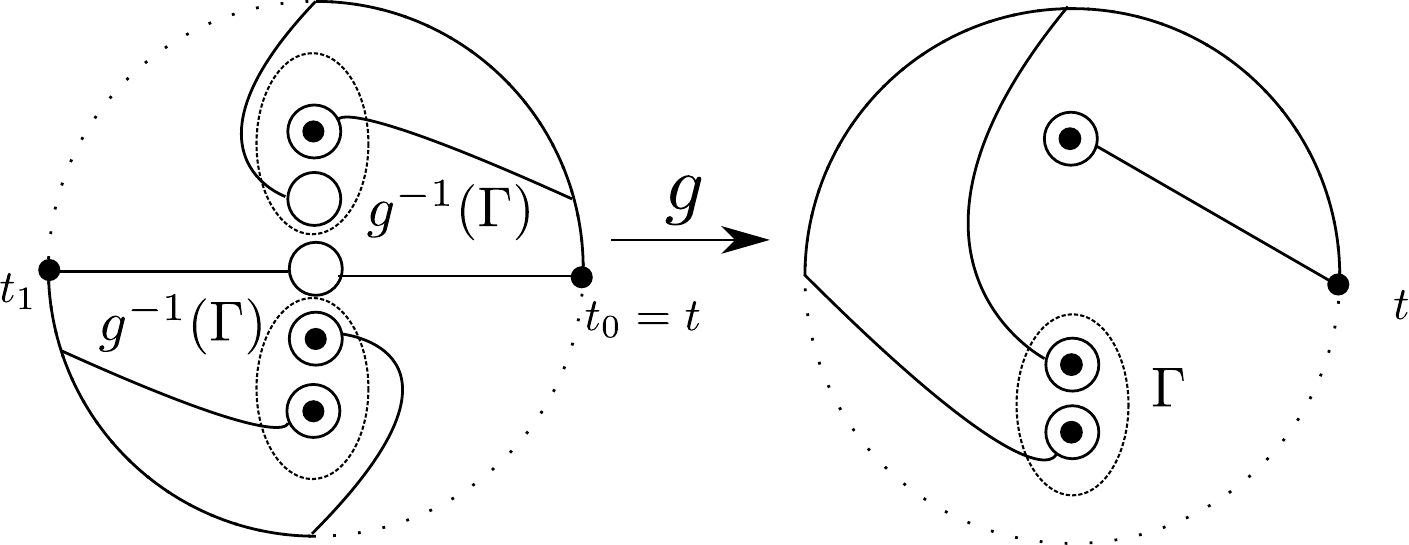}
\caption{Example Levy cycle}
\label{fig:obstructions}
\end{figure}

So a topological polynomial is equivalent to a complex polynomial unless it admits a Levy cycle. But even for topological polynomials, we do not have an algorithm to determine whether an obstruction exists for a particular polynomial. However, some progress has been made.

\begin{definition}
\label{def:hyperbolicpoly}
We say that a topological polynomial $f$ is \emph{hyperbolic} if for all $\omega \in C_f$, there exists some $k \geq 1$ such that $f^{\circ k}(\omega) \in C_f$.
\end{definition}

\begin{definition}
\label{def:periodicpoly}
We say that a topological polynomial $f$ is \emph{periodic} if $C_f \subset P_f$, and \emph{preperiodic} otherwise.
\end{definition}

Notice that a periodic polynomial is always hyperbolic.

In his thesis, Levy used a result of Berstein to prove the following:

\begin{theorem}\cite{Lev85}
\label{thm:levy}
If $f$ is a hyperbolic topological polynomial, then $f$ does not admit a Levy cycle, and so $f$ is equivalent to a complex polynomial.
\end{theorem}

The topological polynomial $f$ in Figure \ref{fig:examplepolys} is periodic, and thus by Berstein-Levy must be equivalent to some complex polynomial.

\section{Mapping schemes}
\label{sec:mapping schemes}

Here we give the standard definitions for mapping schemes as found in \cite{BBLPP00}, but we also borrow some notation from \cite{Koc07}.

\begin{definition}
\label{def:mapping scheme}
A \emph{polynomial mapping scheme of degree $d$} is a tuple $S(C, P, \alpha, \nu)$ where:
\begin{itemize}
\item $Z= C \cup P$ is a finite set,
\item $\alpha : Z \to P$ is surjective,
\item $\nu : Z \to \N$ has $\nu^{-1}(\{ n \geq 2\}) = C$,
\end{itemize}
such that the following conditions are satisfied:
\begin{itemize}
\item The Riemann-Hurwitz Formula: \[d = \frac{\sum_{z \in Z} (\nu(z) - 1)}{2} + 1,\]
\item Local degrees: $\forall z \in Z,$ \[\sum_{\alpha(x) = z} \nu(x) \leq d,\]
\item Infinity:  $\exists z \in C$, which we will call $\infty$, such that: \[\alpha(z) = z, \nu(z) = d.\]
\end{itemize}
\end{definition}

We will treat a mapping scheme as a finite directed graph with vertex set $Z = C \cup P$ and for each $z \in Z$, there are $\nu(z)$ directed edges from $z$ to $\alpha(z)$. A directed cycle in this graph that contains an element of $C$ is called an \emph{attractor}.

\begin{definition}
\label{def:hyperbolicmapping scheme}
A mapping scheme is \emph{hyperbolic} if for all $z \in C$, there exists some $k \geq 1$ such that $\alpha^{\circ k}(z) \in C$ (i.e. if every directed cycle is an attractor).
\end{definition}

\begin{definition}
\label{def:periodicmapping scheme}
A mapping scheme is \emph{periodic} if $C \subset P$ and \emph{preperiodic} otherwise.
\end{definition}

Figure \ref{fig:examplemapping schemes} shows two example mapping schemes of degree 2. The left one is hyperbolic and the right one is not.

\begin{figure}[h]
\includegraphics[width=5in]{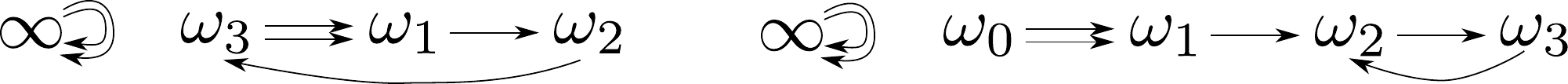}
\caption{Example mapping schemes}
\label{fig:examplemapping schemes}
\end{figure}

\begin{definition}
\label{def:realize}
We say that a topological polynomial $f$ \emph{realizes} a polynomial mapping scheme $S(C, P, \alpha, \nu)$ if there exists a bijection $\beta : C \cup P \to C_f \cup P_f$ such that for all $z \in C \cup P$, we have $f(\beta(z)) = \beta(\alpha(z))$ and the local degree of $f$ at $\beta(z)$ equals $\nu(z)$.
\end{definition}

Notice that the topological polynomials $f$ and $g$ from Figure \ref{fig:examplepolys} realize the mapping schemes in Figure \ref{fig:examplemapping schemes} where $s_i$ loops around $\omega_i$ and $\omega_0 = g^{-1}(\omega_1)$ ($f$ realizes the hyperbolic mapping scheme and $g$ the non-hyperbolic one).

One can easily show that equivalent topological polynomials have isomorphic mapping schemes.

A result of Thom gives the following:

\begin{theorem}\cite{BBLPP00}
For every polynomial mapping scheme, there is a topological polynomial which realizes it.
\end{theorem}

So we can interpret the Berstein-Levy Theorem as a result about mapping schemes:

\begin{theorem}\cite{Lev85}
\label{thm:blschemes}
A hyperbolic polynomial mapping scheme is realizable only by topological polynomials that are equivalent to complex polynomials.
\end{theorem}

In other words, a polynomial mapping scheme with every period being an attractor (i.e. containing an element of $C$) cannot be realized by an obstructed topological polynomial.

It is relatively easy to show that a polynomial mapping scheme with only a single finite period, and that period having length equal to one, cannot be realized by an obstructed topological polynomial.

So naturally we ask which other mapping schemes can be realized by obstructed topological polynomials.

In an unpublished result, Koch found that every unicritical (i.e. $\#(C \setminus \{\infty\}) = 1$) preperiodic polynomial mapping scheme with period length $n \geq 2$ \emph{is} realized by an obstructed topological polynomial \cite{Koch}. Extending these methods, one could establish topological arguments for cases (1) and (2) of our main result, which we restate below. However, we know of no topological constructions for cases (3) or (4).

\begin{main}

Suppose that a polynomial mapping scheme satisfies one of the following conditions:

\begin{enumerate}
\item at least one (non-attractor) period of length at least two and not containing critical values, 
\item at least two (non-attractor) periods not containing critical values,
\item at least two non-attractor periods both of length at least two, or
\item at least four non-attractor periods.
\end{enumerate}

Then this scheme is realized by a topological polynomial that is not equivalent to any complex polynomial.

\end{main}

\section{Automata and bimodules}
\label{sec:automata}

In this section we give some of the standard definitions and results in the theory of self-similar groups (see \cite{Nek05} for a complete introduction to this theory) and introduce some of the key concepts from \cite{Nek09}.

\begin{definition}
\label{def:automaton}
An \emph{automaton} $A$ over an alphabet $X$ is given by a set of input states $A$, a set of output states $B$, and a transition map $\tau : A \times X \to X \times B$.
\end{definition}

If $\tau(a, x) = (y,b)$, then we write $a \cdot x = y \cdot b$ and use notation $y = a(x)$ and $b = a|_x$. We say that $b$ is the \emph{restriction} of $a$ at $x$.

\begin{definition}
\label{def:grpautomaton}
We say that an automaton is a \emph{group automaton} if for every $a \in A$ the map $x \mapsto a(x)$ is a permutation of the alphabet (we assume for group automatons that there exists trivial state $\1\in A, B$ such that $\1\cdot x = x \cdot \1$ for all $x \in X$).
\end{definition}

We represent a group automaton as a labeled directed graph called an \emph{abbreviated Moore diagram} with vertex set equal to the states $A \cup B$ and with a directed edge from $a$ to $b$ labeled by $x$ if and only if $a \cdot x = y \cdot b$ for some $y \in X$. We also label the states by the permutations they induce on $X$. For simplicity, we omit the trivial state $\1$.

In Figure \ref{fig:grigorchuk} we give an example of the abbreviated Moore diagram of the group automaton associated with Grigorchuk's group. 

\begin{figure}[h]
\includegraphics[width=2in]{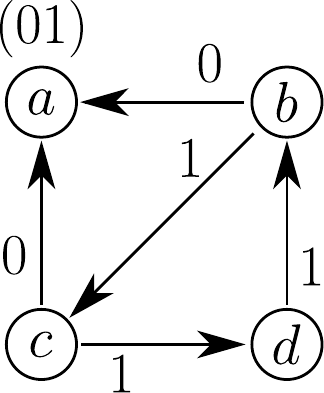}
\caption{Example abbreviated Moore diagram}
\label{fig:grigorchuk}
\end{figure}

A convenient way to describe the action of each of the states of a group automaton is with \emph{wreath recursive notation}. For alphabet $X = \{0,1, ..., d-1\}$, we represent a state $a$ by $\pi_a (a|_{0}, a|_{1}, ... , a|_{d-1})$ where the restrictions are as defined above and $\pi_a$ is the element of the symmetric group of $X$ induced by the action of $a$ (i.e. $a(x) = \pi_a(x)$ for all $x\in X$). We will omit $\pi_a$ when it is trivial (we say such a state is \emph{inactive}) and we omit the restrictions if they are all the trivial state $\1$.

The wreath recursive notations for the states of the Girgorchuk automaton in Figure \ref{fig:grigorchuk} are: $a = (01), b = (a, c), c = (a,d), d = (\1, b)$.

We can think of $X^*$ (the set of words in the finite alphabet $X$) as an infinite, rooted, $d$-ary tree (where $d$ is the size of the alphabet). The root vertex is the empty word, the first level of vertices are the letters of $X$, and each word $w \in X^*$ is adjacent to the $d$ vertices in the level below it of the form $wx$ for $x \in X$. We present the beginning of the binary tree for $X = \{0,1\}$ in Figure \ref{fig:tree}.

\begin{figure}[h]
\includegraphics[width=4in]{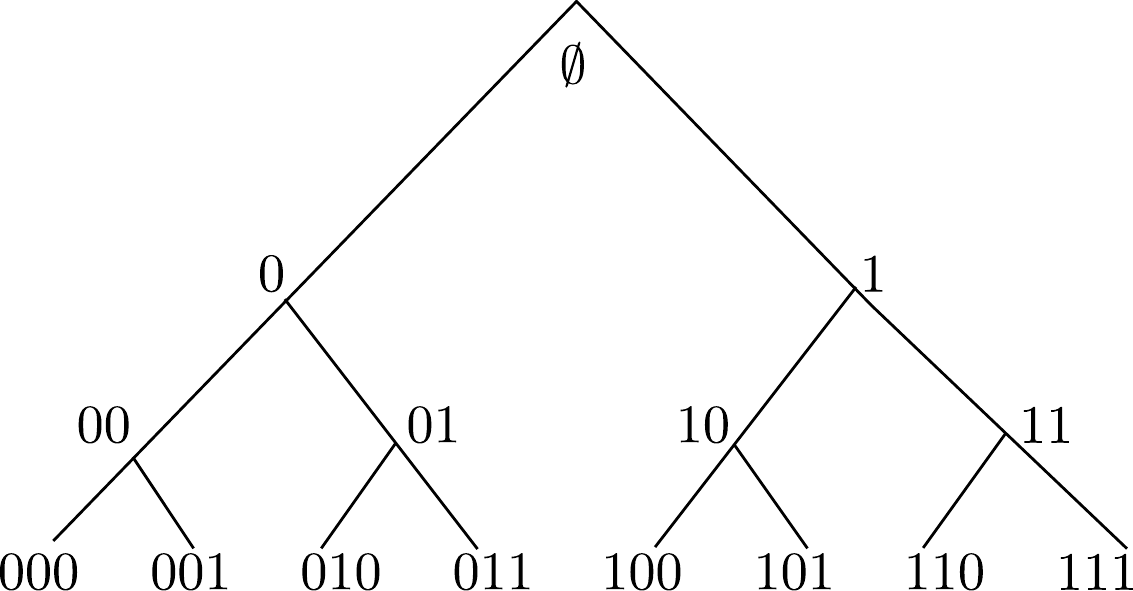}
\caption{The first three levels of the binary tree}
\label{fig:tree}
\end{figure}

When $B \subset A^*$, we have an action of $F(A)$ (the free group on the input states $A$) on the tree $X^*$ by graph automorphisms as follows. For $a \in A$, we have the action on the word $w = xw' \in X^*$ given by $a(xw') = a(x)a|_x(w')$. Notice that $a|_x \in B \subset A^*$, so the action of $a|_x$ on $w'$ is well-defined. The wreath recursion $a = \pi_a(a|_{0}, a|_{1}, ... , a|_{d-1})$ describes the automorphism of $X^*$ in the following way: $\pi_a$ gives the action of $a$ on the first level of the tree, and the restrictions $a|_{i}$ describe the actions on the subtrees (which are isomorphic to the entire tree).

We can compute the wreath recursive notation for the product (as elements of $F(A)$) of $a,b \in A$ where \\ $a = \pi_a(a|_{0}, a|_{1}, ... , a|_{d-1})$ and $b = \pi_b(b|_{0}, b|_{1}, ... , b|_{d-1})$ by: $$ab = \pi_a\pi_b (a|_{\pi_b(0)}b|_{0}, a|_{\pi_b(1)}b|_{1}, ... , a|_{\pi_b(d-1)}b|_{d-1}).$$

\begin{definition}
For $G$ and $H$ groups, a \emph{permutational ($G$-$H$)-bimodule} is a set $\M$ with a left action of $G$ and a right action of $H$ which commute. A \emph{covering bimodule} has free right action with only finitely many orbits. We call a ($G$-$G$)-bimodule simply a $G$-bimodule.
\end{definition}

We say that two ($G$-$H$)-bimodules are isomorphic if there exists a bijection between them that agrees with the actions (that is, a bijective map $F$ such that for all $g \in G, h \in H$ and $x$ in the domain bimodule we have $F(g \cdot x \cdot h) = g \cdot F(x) \cdot h$).

For $\M_1$ a ($G_1$-$G_2$)-bimodule and $\M_2$ a ($G_2$-$G_3$)-bimodule, we may form the \emph{tensor product} $\M_1\otimes \M_2$ which is the ($G_1$-$G_3$)-bimodule equal to the quotient of $\M_1 \times \M_2$ by the equivalence $$(x_1 \cdot g_2) \otimes x_2 = x_1 \otimes (g_2 \cdot x_2)$$ for all $g_2 \in G_2, x_1 \in \M_1, x_2 \in \M_2$. 
The actions are defined in the natural way: $g_1 \cdot (x_1 \otimes x_2) \cdot g_3 = (g_1 \cdot x_1) \otimes (x_2 \cdot g_3)$ for all $g_1 \in G_1, g_3 \in G_3, x_1 \in \M_1, x_2 \in \M_2$.

It is straightforward to show that the tensor product of bimodules is an associative operation and that the tensor product of covering bimodules is again a covering bimodule.

For $\M$ a covering ($G$-$H$)-bimodule, a \emph{basis} of $\M$ is an orbit transversal $X$ to the right action. 
So if $X$ is a basis, every element of $\M$ can be written uniquely as $y \cdot h$ for some $y\in X, h \in H$. 
Thus, for every $g\in G, x \in X$, we have that there exists a unique pair $y\in X, h \in H$ such that $g \cdot x = y \cdot h$. 

Notice that if we have a covering ($G$-$H$)-bimodule $\M$ with basis $X$, we may construct an abstract automaton with set of input states $G$ and set of output states $H$ over the alphabet $X$, where for any $g\in G, x \in X$ and $y \in X, h\in H$ the unique pair such that $g \cdot x = y \cdot h$, we set $g(x) = y$ and $g|_x = h$. We call this the \emph{complete automaton} for the bimodule. We can similarly define an automaton using as input states a generating set of $G$; we say that such an automaton \emph{generates} $\M$. Likewise, given a finite group automaton with $B \subset A^*$, we may define the $F(A)$-bimodule that it generates via the action on $X^*$ described previously.

\begin{definition}
\label{def:cyclediagram}
For $A = (a_i)_{i\in I}$ a sequence of permutations of a finite set $X$ (we do not use a set of permutations since we wish to allow for repeated elements), we define the \emph{cycle diagram} of $A$ to be the oriented 2-dimensional CW-complex $D(A)$ with 0-cells the elements of $X$ and for every cycle $(x_1, x_2, ... , x_n)$ of every permutation of $A$, we attach a 2-cell to the vertices $x_1, x_2, ... , x_n$ so that their order around the boundary of the 2-cell matches the order in the cycle. Two different 2-cells do not have any 1-cells in common.
\end{definition}

Figure \ref{fig:examplecycles} shows three example cycle diagrams. The first is for the permutations $(1 2 3 4), (1 2)(3 4)$, the second is for the permutations $(1 2 3), (1 3 4)$, and the third is for the permutations $(1 2)(3 4), (1 4)$.

\begin{figure}[h]
\includegraphics[width=4in]{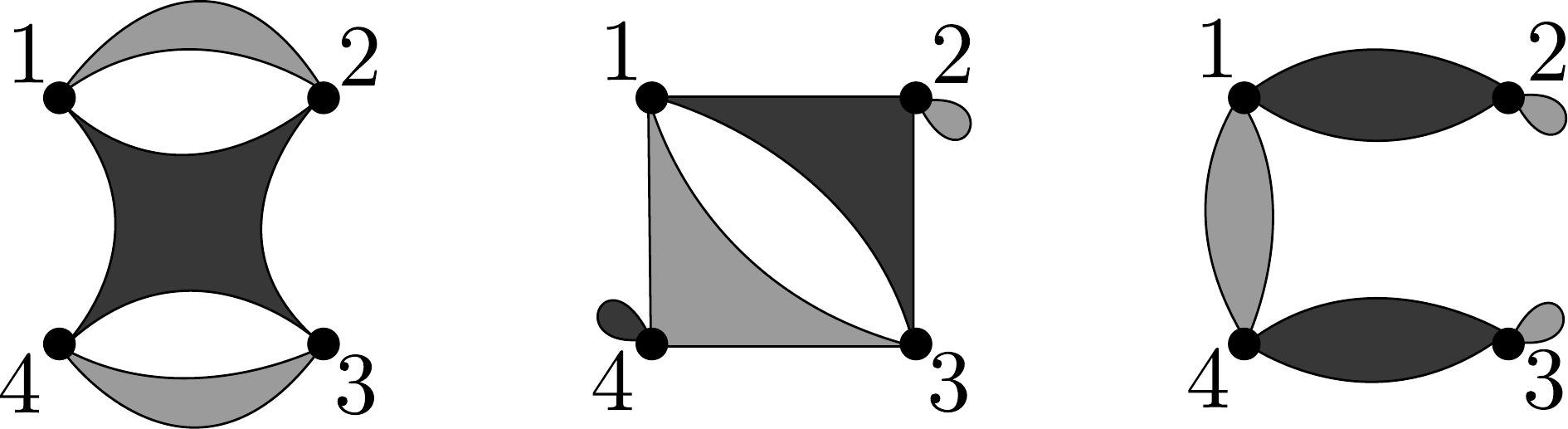}
\caption{Example cycle diagrams}
\label{fig:examplecycles}
\end{figure}

We say a sequence $A = (a_i)_{i \in I}$ of permutations of a finite set $X$ is \emph{dendroid} if its cycle diagram $D(A)$ is contractible. Notice that a dendroid sequence must be transitive, cannot have any non-trivial cycles appear more than once, and must have that any two cycles are disjoint or share only one element.

In the examples in Figure \ref{fig:examplecycles}, the first two are not dendroid, but the third is.

Alternatively, we may define a hypergraph on the vertices $X$ where each cycle of length at least two defines a hyperedge (i.e. a set containing at least two vertices). The sequence of permutations is dendroid if and only if this hypergraph is connected with no cycles.

\begin{definition}
\label{def:dendroid}
We say that a group automaton with alphabet $X$, set of input states $A$, and set of output states $B$ is \emph{dendroid} if all three of the following conditions hold:

\begin{enumerate}
\item The sequence of permutations on $X$ defined by elements of $A$ is dendroid.
\item For every $b\in B$, there exists a unique pair $a\in A, x \in X$ such that $a \cdot x = y \cdot b$ for some $y \in X$.
\item For every element $a\in A$ and every cycle $(x_1, x_2, ... , x_n)$ of the action of $a$ on $X$, we have that $a|_{x_i} =  \1$ for all but at most one index $i$.
\end{enumerate}

\end{definition}

Notice that the Grigorchuk automaton in Figure \ref{fig:grigorchuk} is \emph{not} dendriod becuase it violates condition (2) above (the state $a$ has two incoming arrows). In Figure \ref{fig:exampleautomata} we give two examples of abbreviated Moore diagrams of dendroid automata on the binary alphabet.

\begin{figure}[h]
\includegraphics[width=4in]{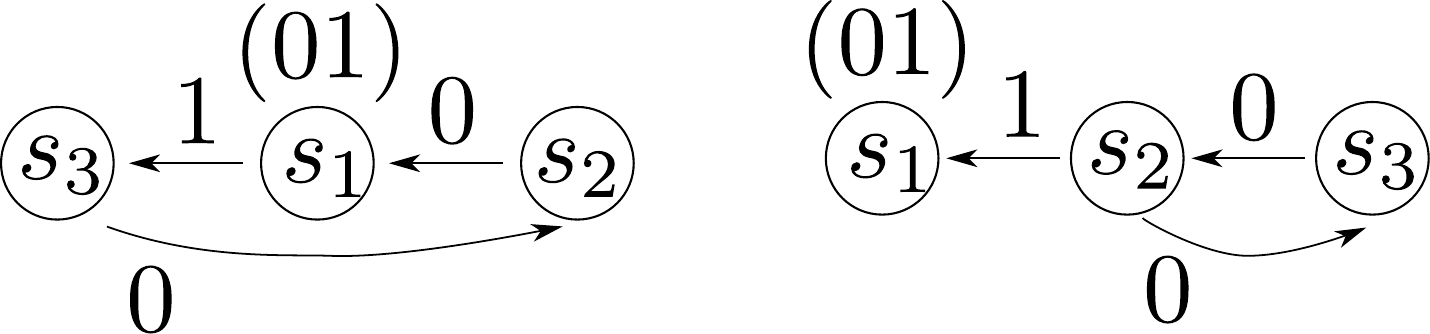}
\caption{Example dendroid automata}
\label{fig:exampleautomata}
\end{figure}

In fact, the automata in Figure \ref{fig:exampleautomata} are not only dendroid, but also satisfy the stronger definition of \emph{kneading}:

If we place a cyclic ordering $(a_1, a_2, ... , a_n)$ on the input set of states $A$ of a dendroid automaton, we get an induced cyclic ordering $(b_1, b_2, ... , b_m)$ on the output set of states $B$. A dendroid automaton is called a \emph{twisted kneading automaton} if it has cyclically ordered input set $A = (a_1, a_2, ... , a_n)$ and output set $B$ equal to conjugates of elements of $A$ with induced cyclic ordering $(a_1, a_2, ... , a_n)^\alpha$ for some $\alpha \in B_n$, the braid group on $n$ strands. The action of the generators $\sigma_i \in B_n$ is by $(a_1, ... , a_i, a_{i+1}, ... a_n)^{\sigma_i} = (a_1, ... , a_{i+1}, a_i^{a_{i+1}}, ... , a_n)$ where $a_i^{a_{i+1}}= a_{i+1}^{-1}a_ia_{i+1}$. A twisted kneading automaton with $\alpha$ trivial is simply called a \emph{kneading automaton}.

In other words, a kneading automaton is dendroid with set of output states $B$ equal to set of input states $A$. A twisted kneading automaton has $B$ instead equal to the image of $A$ under a ``twist'' by an element of the braid group. We will refer to a bimodule generated by a (twisted) kneading automaton as a (twisted) kneading bimodule.


Recall that the examples in Figure \ref{fig:exampleautomata} are kneading automata. In Figure \ref{fig:twistedautomata} we give examples of twisted kneading automata with non-trivial twist.

\begin{figure}[h]
\includegraphics[width=4in]{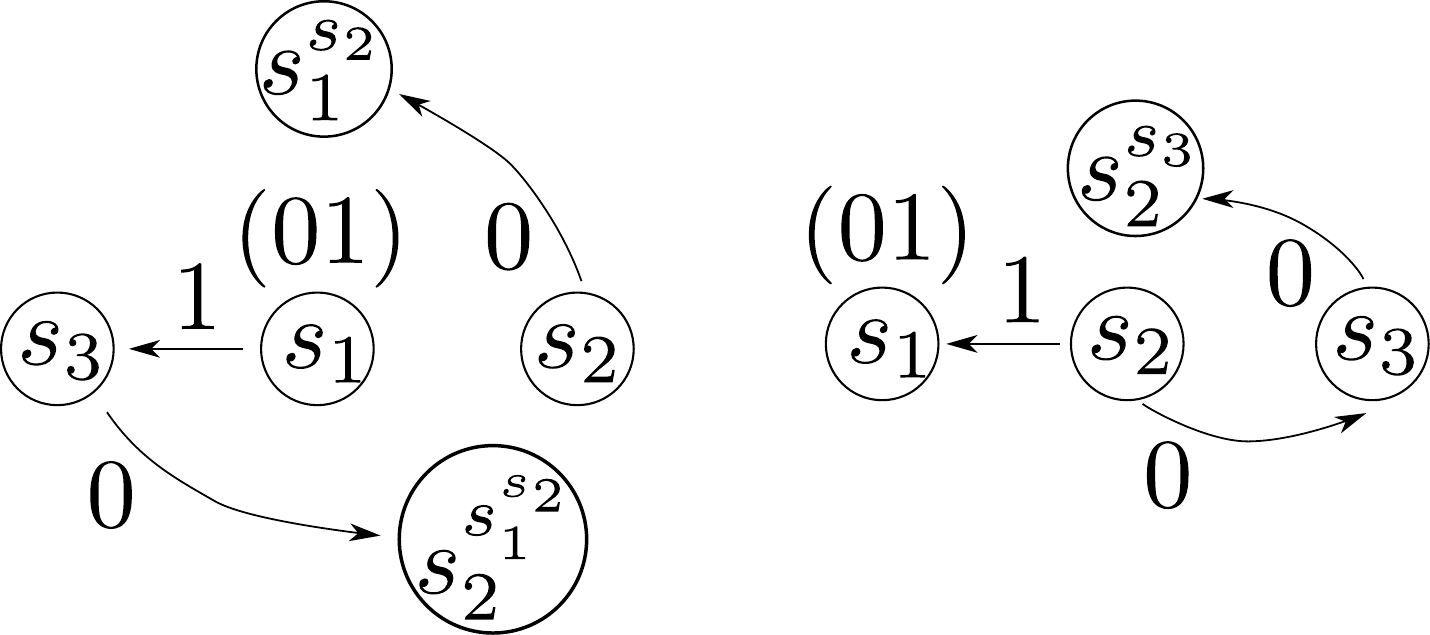}
\caption{Example twisted kneading automata}
\label{fig:twistedautomata}
\end{figure}

\section{Nekrashevych's characterization of topological polynomials}
\label{sec:Nekrashevych}

In this section we continue to summarize the definitions and results of \cite{Nek09}.

For $f : S^2 \to S^2$ a topological polynomial with post-critical set $P_f = \{\infty, \omega_1, ... , \omega_{n}\}$, let $\{s_i\}_{i=1}^n$ be a planar generating set of $\pi_1(S^2 \setminus P_f, t)$, for $t \in S^2 \setminus P_f$ the basepoint on the circle at infinity (see Section \ref{sec:thurston}). That is, on the closed disc that is a retraction of $S^2 \setminus \{\infty\}$, the generator $s_i$ is a simple loop based at $t$ going around $\omega_i$ in the positive direction, and the loops are cyclically ordered in the positive direction.

Figure \ref{fig:gensets} gives three examples of different planar generating sets with the same set of punctures.

\begin{figure}[h]
\includegraphics[width=5in]{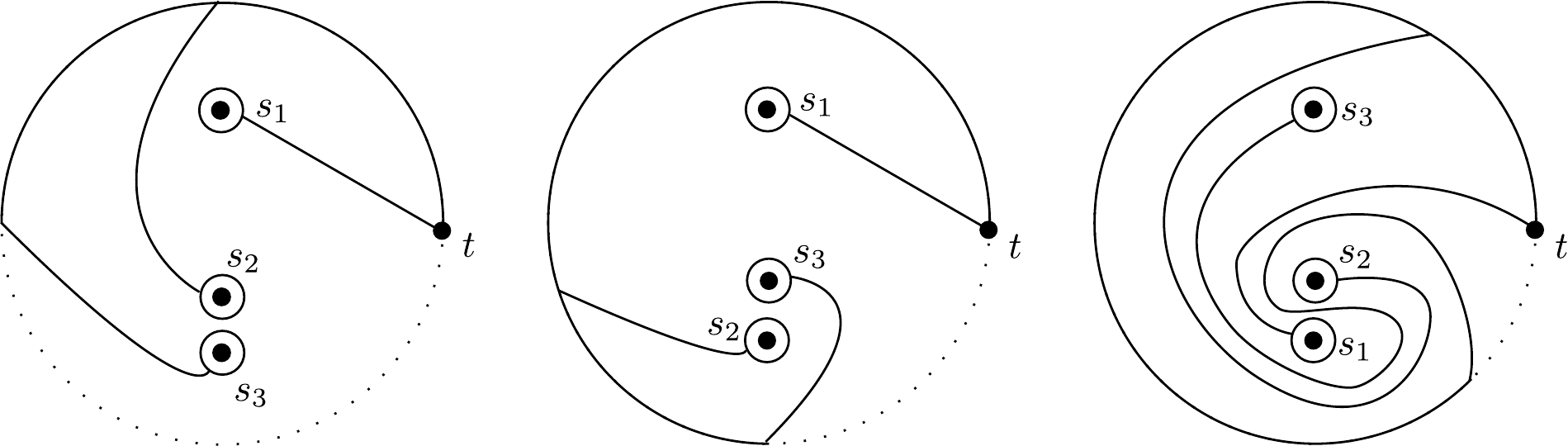}
\caption{Example planar generating sets}
\label{fig:gensets}
\end{figure}

Let $t_j \in f^{-1}(t) = \{t = t_0, t_1, ... , t_{d-1}\}$. Define $\M_f$ to be the $(\pi_1(S^2 \setminus P_f, t)-\pi_1(S^2\setminus P_f, t_j))$-bimodule of homotopy classes of paths in $S^2 \setminus P_f$ starting at $t_j$ and ending at any point in $f^{-1}(t)$. The right action of $\pi_1(S^2\setminus P_f, t_j)$ is by concatenation of the loop at $t_j$ to the beginning of the path, and the left action of $\pi_1(S^2\setminus P_f, t)$ is by concatenation of the $f$-lift of the loop at $t$ to the end of the path.

Up to isomorphism, $\M_f$ does not depend on the choice of basepoints. If we choose a path from $t$ to $t_j$, we may identify the fundamental groups in the usual way, and the isomorphism class of the $\pi_1(S^2\setminus P_f)$-bimodule $\M_f$ does not depend on the choice of this path. So we will choose $j=0$ and so identify the fundamental groups by the constant path at $t=t_0$.

Notice that the right action is free and that two elements of $\M_f$ belong to the same right orbit if and only if they have the same endpoints. So the number of right orbits equals the degree of $f$, and the bimodule is covering.

We make a canonical choice for a basis of $\M_f$, $X= \{0, 1, ... , d-1\}$. The basis element $k$ for the orbit associated with the endpoint $t_k$ is the path in the positive direction along the circle at infinity from $t$ to $t_k$.

A result of Nekrashevych states that the combinatorial data encoded in this bimodule completely describes the post-critical dynamics of the topological polynomial, up to Thurston equivalence.

\begin{theorem}\cite{Nek09}\label{thm:Nek1}
The bimodule $\M_f$ defined above is generated by a twisted kneading automaton, and the twisted kneading automaton associated to the topological polynomial $f$ along with the cyclic order $(s_1, s_2, ... , s_n)$ of the generators of $F_n = \pi_1(S^2 \setminus P_f)$ uniquely determine the Thurston equivalence class of $f$.
\end{theorem}

In fact, these bimodules can be described more explicitly.

For $F_{n} = F(s_1, s_2, ... , s_{n})$ the free group on $n$ generators, define $a_{i, j} \in Aut(F_{n})$ by $a_{i, j}(s_i) = s_i^{s_j}$ and for all $k \neq i$, $a_{i, j}(s_k) = s_k$. Notice that $a_{j, j}$ and $[a_{i, j}, a_{k, j}]$ are trivial for all $1 \leq i, j, k \leq n$, $[a_{i,j}, a_{k,l}]$ is trivial for all $i,j,k,l$ distinct, and for a fixed $j_0$ we have $\prod_{1 \leq i \leq n} a_{i, j_0} \in Inn(F_{n})$. For $q: Aut(F_{n}) \to Out(F_{n})$ the quotient map, let $P\Sigma O_n$ be the image under $q$ of $P\Sigma_n = \langle a_{i, j} \rangle \leq Aut(F_{n})$. We call $P\Sigma O_n$ the \emph{pure symmetric outer automorphism group} of the free group of rank $n$. From now on, we will abuse notation and write $a_{i,j}$ for its image in $Out(F_n)$.

For $G$ a group and $\alpha \in Aut(G)$, we define the associated $G$-bimodule $[\alpha]$ to be the set of expressions $\alpha \cdot g$ for $g \in G$ with the actions $h_l \cdot (\alpha \cdot g) \cdot h_r = \alpha \cdot \alpha(h_l) gh_r$ for all $h_l, h_r \in G$. It is easy to show that for $\alpha \in Inn(G)$ and $\M$ any $G$-bimodule, $[\alpha] \otimes \M \simeq \M \simeq \M \otimes [\alpha]$, so we can uniquely define the isomorphism class of the bimodule for $\alpha \in Out(G)$.

What Nekrashevych actually showed in the proof of Theorem \ref{thm:Nek1} is that the bimodule $\M_f$ is isomorphic to a twisted kneading bimodule of the form $\M_{\K(f)} \otimes [\phi]$ where $\M_{\K(f)}$ is a kneading bimodule and $\phi \in P\Sigma O_n$. He went on to define the $P\Sigma O_n$-bimodule $\mathfrak{G}_f = \{[\alpha] \otimes \M_f \otimes [\beta] \ | \  \alpha, \beta \in P\Sigma O_n \}$, with the natural left and right actions. Nekrashevych proved the following:

\begin{proposition}\cite{Nek09}\label{thm:Nek2}
Every twisted kneading automaton over $F_n$ is associated with some post-critically finite topological polynomial.
\end{proposition}

We can think of $\alpha, \beta \in P\Sigma O_n$ as acting on $S^2 \setminus P_f$ and the tensor operation as functional composition. Figure \ref{fig:aij} demonstrates the action of $a_{i,j}$ on $S^2 \setminus P_f$ and the planar generating set $\{s_1, ... , s_{n}\}$.

\begin{figure}[h]
\includegraphics[width=5in]{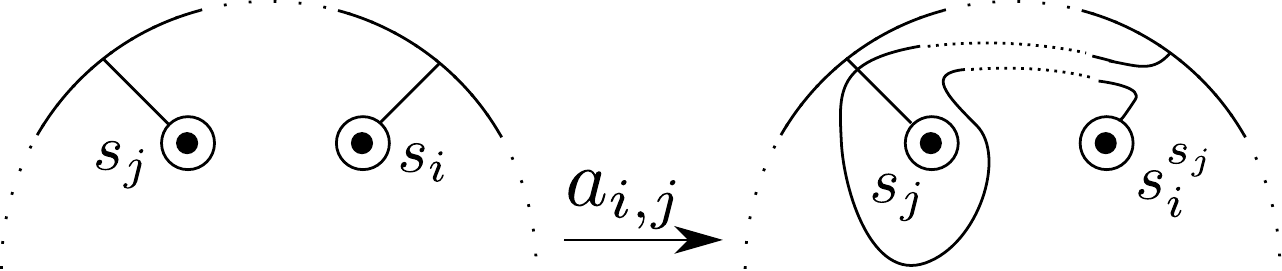}
\caption{Action of $a_{i,j}$}
\label{fig:aij}
\end{figure}

\begin{lemma}\label{lem:twist}
The action of $a_{i+1,i}a_{i,i+1}$ is that of a Dehn Twist about a curve separating $\{\omega_i, \omega_{i+1}\}$ from the rest of $P_f$ ($\omega_i$ is the puncture surrounded by $s_i$).
\end{lemma}

\begin{proof}
See Figure \ref{fig:twist}.
\end{proof}

\begin{figure}[h]
\includegraphics[width=5in]{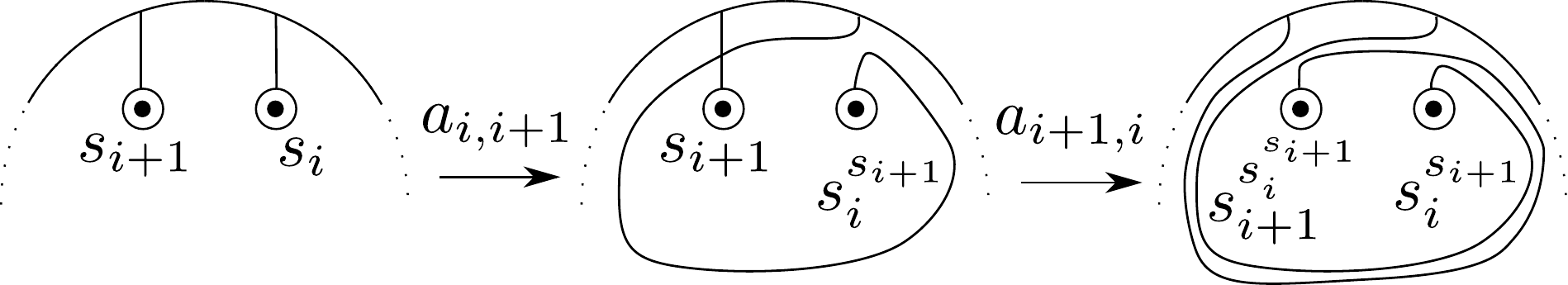}
\caption{Action of $a_{i+1,i}a_{i,i+1}$ by a Dehn twist}
\label{fig:twist}
\end{figure}

Further, we may represent a Dehn Twist about a curve separating $\{\omega_i, \omega_{i+1}, \omega_{i+2}\}$ from the rest of the post-critical set by the word $a_{i+2, i+1}a_{i+2, i}a_{i+1, i}a_{i+1, i+2}a_{i, i+2}a_{i, i+1} \in P \Sigma O_{n}$. And so on.

In this way, we see that the pure mapping class group $PMod(S^2 \setminus P_f)$ is a subgroup of $P\Sigma O_n$. In fact, the $P\Sigma O_n$-bimodule $\mathfrak{G}_n$ was originally defined by Bartholdi and Nekrashevych as a $PMod(S^2 \setminus P_f)$-bimodule \cite{BN06} and later extended by Nekrashevych \cite{Nek09}.

The twisted kneading bimodule $\M_{\K(f)} \otimes [\phi]$ encodes a topological description of the preimages of the planar generating set $\{s_i\}_{i=1}^n$ under the topological polynomial $f$. The kneading bimodule $\M_{\K(f)}$ is generated by a kneading automaton $\K(f)$ determined (up to labels) by the mapping scheme of $f$, which we will denote by $S(f)$. The kneading automaton $\K(f)$ has a state for each generator $s_i \in \pi_1(S^2 \setminus P_f)$. We have an arrow from state $s_i$ to state $s_j$ if and only if $f(\omega_j) = \omega_i$. So the unlabeled Moore diagram of this kneading automaton is the subgraph of $S(f)$ induced by the vertices in $P_f$, but with the arrows reversed.

The labels are determined by numbering the lifts of the basepoint $t$. Recall the alphabet $X = f^{-1}(t) = \{0, 1, ... , d-1\}$ for $d = deg(f)$. We label the directed edge from $s_i$ to $s_j$ by $k$ if $\omega_j$ is contained in the loop that goes from $t$ to $t_k$ along the circle at infinity in the positive direction, then follows the lift of $s_i$ starting at $t_k$, and then travels back along the circle at infinity in the negative direction to $t$. The active states will be those associated with critical values of $f$. For $\omega_i$ a critical value of $f$, we have $s_i(k) = k'$ if the $f$-lift of $s_i$ that starts at $t_k$ ends at $t_{k'}$.

Notice that the permutations of an active state and the coordinates of the non-trivial restrictions are related by the type of critical point(s) with which the corresponding critical value is associated.

\begin{lemma}
\label{thm:cycles}

For $\omega_0$ a critical value of topological polynomial $f$ and $\{ \omega_1, \omega_2, ... , \omega_k\} = f^{-1}(\omega_0)$ with the local degree of $f$ at $\omega_i$ equal to $d_i$ for $1 \leq i \leq k$, then:

\begin{enumerate}
\item The action of the state $s_0$ in the kneading automaton $\K(f)$ will have $k$ cycles, and their lengths will be given by the set $\{d_i\}_{i=1}^{k}$.

\item If $\omega_i \in P_f$, then the label of the arrow from state $s_0$ to $s_i$ (the state corresponding with $\omega_i \in P_f$) will be one of the coordinates on which $s_0$ acts by a cycle of length $d_i$. 

\item Any labels of edges in $\K(f)$ from $s_0$ to states not associated with critical points will be of coordinates on which $s_0$ acts trivially.

\end{enumerate}

\end{lemma}

\begin{proof}

These properties follow more or less immediately from the definitions of $\M_f$ and $\K(f)$.

\begin{enumerate}

\item If $s_0$ is a small loop about $\omega_0$ connected by an arc to $t$, then its preimages about $\omega_i$, a critical point with local degree $d_i$ will be a small loop about $\omega_i$ with $d_i$ arcs connecting it to $d_i$ preimages of $t$. Notice that this will give a cycle of these $d_i$ preimages in $S(f)$.

\item Continuing as above, notice that one of the paths from $t$ to its preimage, along the arc to the loop about $\omega_i$, back along the next arc in the positive direction to a different preimage of $t$, and then back to $t$ will contain $\omega_i$ in its interior, and starting at any other preimage will not contain $\omega_i$.

\item Unlike in our first observation, if $s_0$ is a small loop about $\omega_0$ connected by an arc to $t$, then its preimage about a non-critical post-critical point $\omega$ will also consist of a small loop about $\omega$ connected by an arc, but to a preimage of $t$. Thus, the preimage of $s_0$ which starts at this preimage of $t$ (which is the coordinate on which $s_0$ will restrict to the state associated with $\omega$) also terminates at this preimage of $t$.

\end{enumerate}

\end{proof}

Now a lift of $s_i$ might not be homotopic to a generator; it might even surround multiple post-critical points. Expressing these lifts as elements of $F(s_1, ... , s_n)$ allows us to determine the twist $\phi \in P\Sigma O_n$ associated with the polynomial. This will not be important to our work, and so we will only give an example (see below) and refer the reader to \cite{BN06}, \cite{Nek05}, \cite{Nek09} to understand the details of how to compute the element $\phi$.

In Figures \ref{fig:f} and \ref{fig:allg}, we give two examples of starting with a topological polynomial (we use $f$ and $g$ from Figure \ref{fig:examplepolys}), finding its mapping scheme, and then producing the associated kneading automaton. Here we explain how to find $\phi \in P\Sigma O_3$ such that $\M_f \simeq \M_0 \otimes [\phi]$ where $\M_0$ is the bimodule generated by the kneading automaton $\K(f)$:

If we read off the automaton generating $\M_f$ from the map $f$ at the top of Figure \ref{fig:f}, we get the following wreath recursion:

\begin{gather*}
s_1 = (01) (s_2^{-1}s_1^{-1}, s_1s_2s_3)\\
s_2 = (s_1, \1)\\
s_3 = (s_2, \1)
\end{gather*}

Notice that this automaton is not kneading or twisted kneading. Let $\alpha \in Aut(F_3)$ be the conjugation by $s_2s_3$. Since $\alpha \in Inn(F_3)$, we have that $[\alpha] \otimes \M_f \simeq \M_f$. To compute the automaton generating $[\alpha] \otimes \M_f$, we need only conjugate our wreath recursion for $\M_f$ by $s_2s_3 = (s_1s_2, \1)$. This gives us:

\begin{gather*}
s_1 = (01) (\1, s_3)\\
s_2 = (s_1^{s_2}, \1)\\
s_3 = (s_2^{s_1^{s_2}}, \1)
\end{gather*}

This automaton \emph{is} twisted kneading with kneading automaton $\K(f)$ (see the bottom of Figure \ref{fig:f}) and twist $\phi = a_{2,1}a_{1,2}$. The wreath recursion for $\K(f)$ is just the previous one without the conjugations, since we have separated the action of the twist $\phi$:

\begin{gather*}
s_1 = (01) (\1, s_3)\\
s_2 = (s_1, \1)\\
s_3 = (s_2, \1)
\end{gather*}

The reader may wish to check understanding of this process by computing the twist for $g$ in Figure \ref{fig:allg} (the correct twist is $a_{2,3}$).

\begin{figure}[h]
\includegraphics[width=5in]{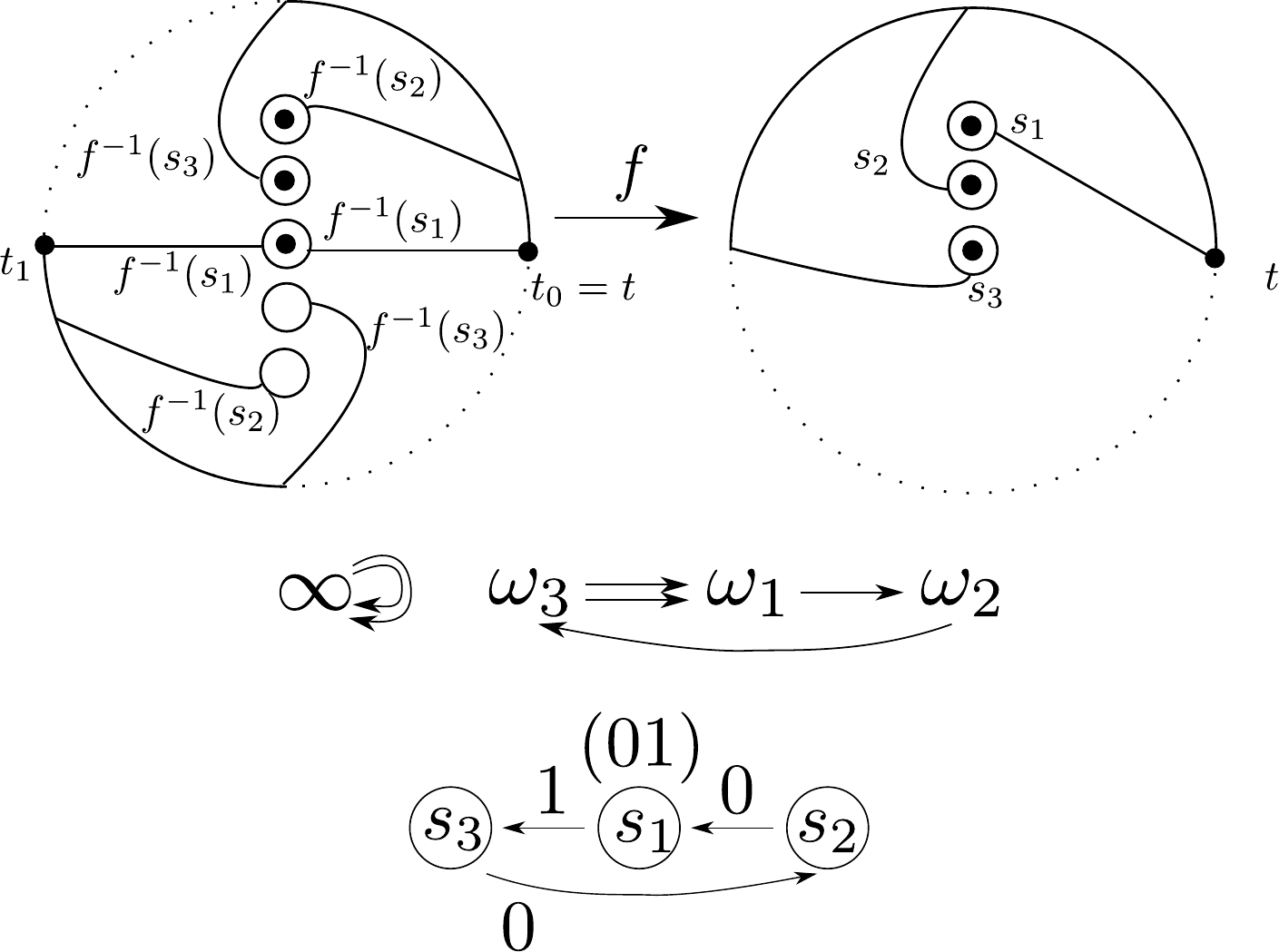}
\caption{The topological polynomial $f$, its mapping scheme $S(f)$, and its kneading automaton $\K(f)$}
\label{fig:f}
\end{figure}

\begin{figure}[h]
\includegraphics[width=5in]{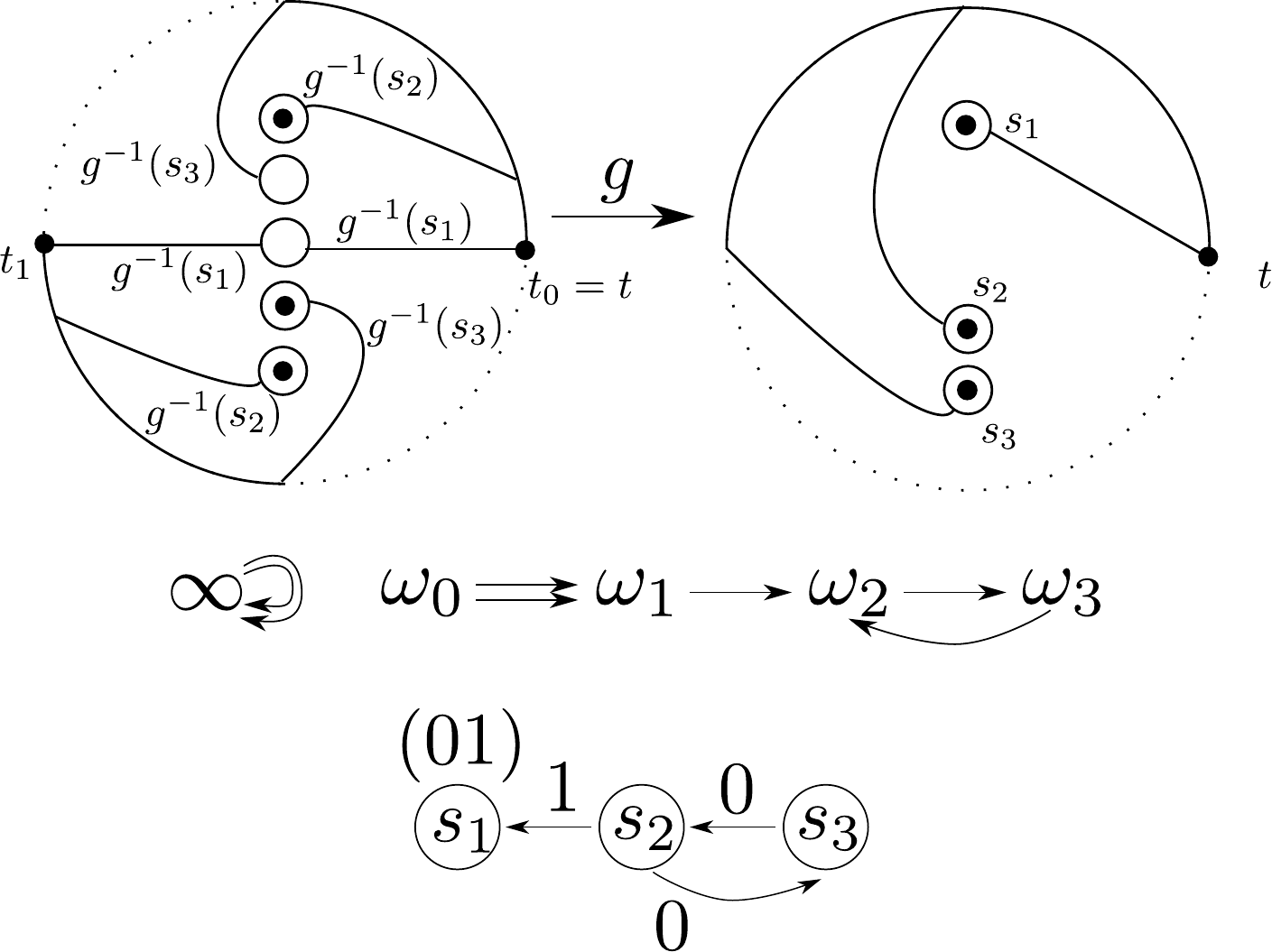}
\caption{The topological polynomial $g$, its mapping scheme $S(g)$, and its kneading automaton $\K(g)$}
\label{fig:allg}
\end{figure}

\section{The quadratic example}
\label{sec:quad}

We will now describe the bimodule $\mathfrak{G}_f$ for $\deg(f) = 2$ and $f$ preperiodic with preperiod $k \geq 1$ and period $n \geq 2$. In other words, $f$ will be a topological polynomial realizing the mapping scheme in Figure \ref{fig:quadmapping scheme}.

\begin{figure}[h]
\includegraphics[width=5in]{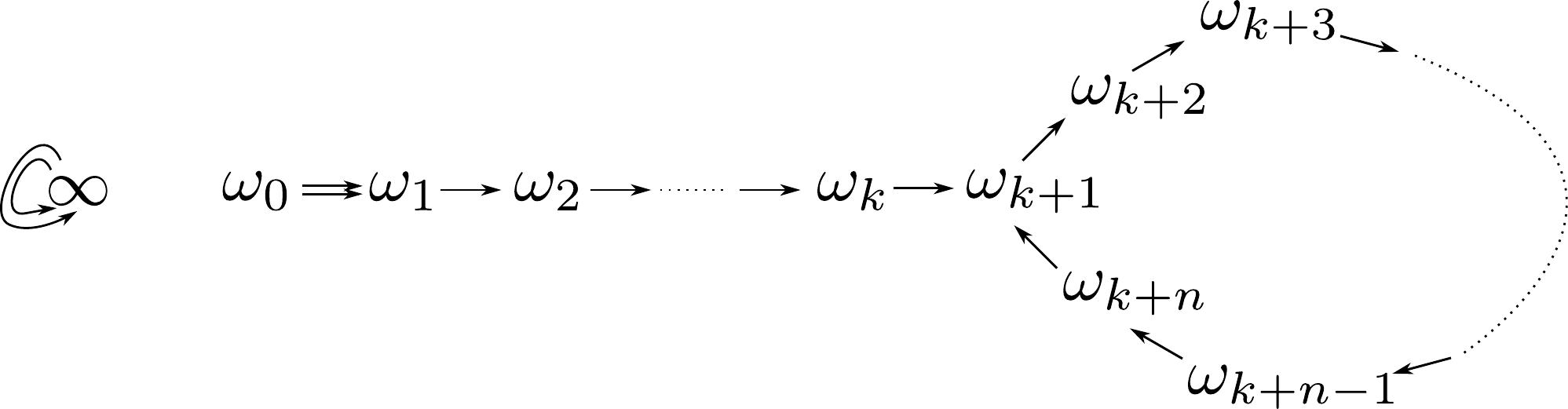}
\caption{The mapping scheme of a preperiodic quadratic topological polynomial}
\label{fig:quadmapping scheme}
\end{figure}

For $x_1x_2 ... x_k ... x_{k+n} \in \{0,1\}^{k+n}$ such that $x_k = \overline{x_{k+n}}$, define $\M_{x_1x_2 ... x_k, x_{k+1} ... x_{k+n}}$ to be the $F_{k+n}$-bimodule generated by the automaton $\K_{x_1x_2 ... x_k, x_{k+1} ... x_{k+n}}$ with states $\{\1, s_1, ... , s_{k+n}\}$ and alphabet $\{0,1\}$ defined by:

\begin{itemize}
\item $s_1 = (01)$
\item $s_{k+1} = (s_{k+n}, s_k)$ if $x_k = 1$ and $x_{k+n} = 0$, and $s_{k+1} = (s_k, s_{k+n})$ if $x_k = 0$ and $x_{k+n} = 1$
\item for all $1 \leq i  < k+n, i \neq k+1$, $s_{i+1} = (s_i, \1)$ if $x_i = 0$, and $s_{i+1} = (\1, s_i)$ if $x_i = 1$.
\end{itemize}

See Figure \ref{fig:quadkneading} for the abbreviated Moore diagram of this automaton.

\begin{figure}[h]
\includegraphics[width=5in]{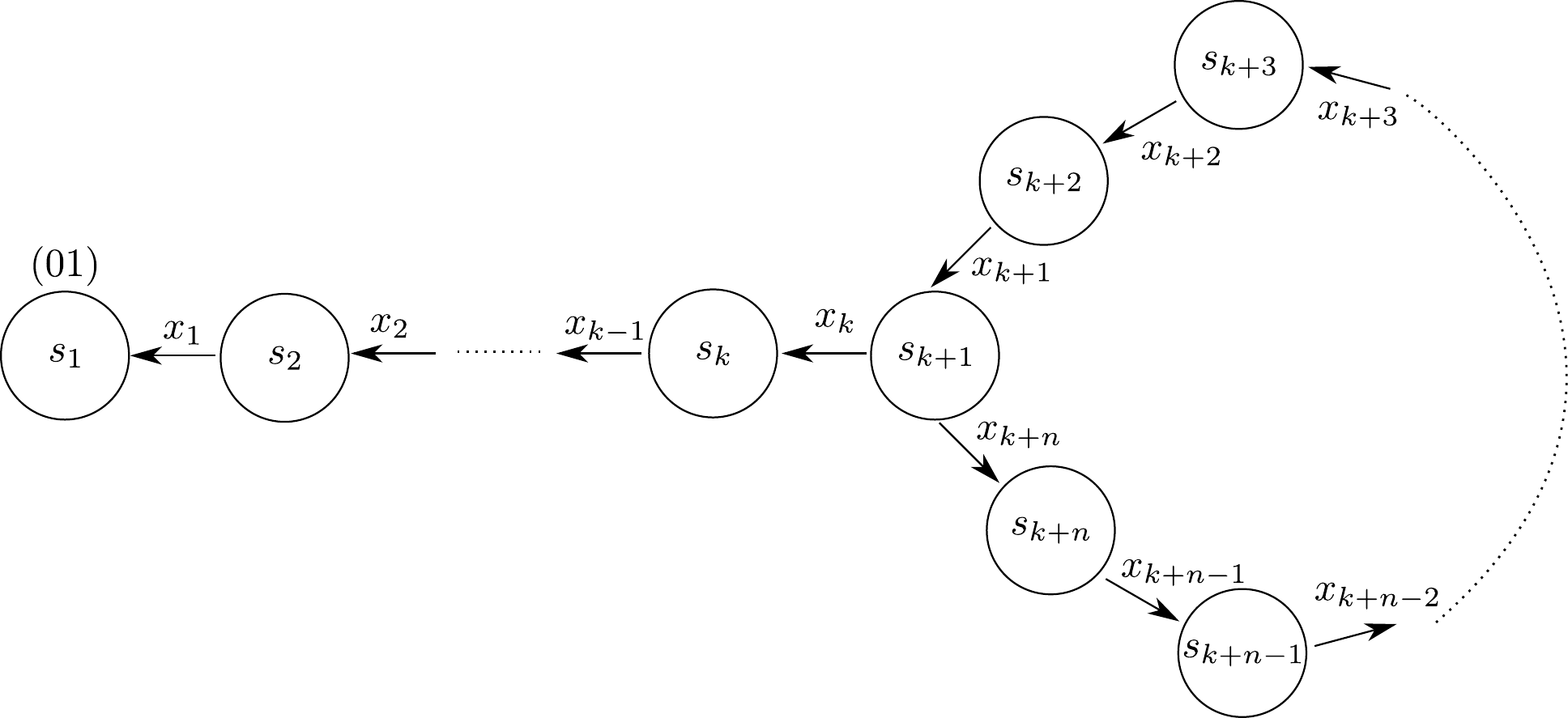}
\caption{Moore diagram of the automaton $\K_{x_1x_2 ... x_k, x_{k+1} ... x_{k+n}}$}
\label{fig:quadkneading}
\end{figure}

The automata $\K_{x_1x_2 ... x_k, x_{k+1} ... x_{k+n}}$ with $x_k = \overline{x_{k+n}}$ are precisely the preperiodic binary kneading automata with preperiod $k \geq 1$ and period $n \geq 2$ \cite{Nek05}.

Notice that $\M_{x_1x_2 ... x_k, x_{k+1} ... x_{k+n}} \simeq \M_{\overline{x_1}\overline{x_2} ... \overline{x_k}, \overline{x_{k+1}} ... \overline{x_{k+n}}}$ by the map that switches 0 and 1, so we will assume $x_k = 1, x_{k+n} = 0$ (i.e. that $s_{k+1} = (s_{k+n}, s_k)$).

We define the $P\Sigma O_{k+n}$-bimodule $\G_{k,n}$ to be the set of all $F_{k+n}$-bimodules of the form $[\alpha] \otimes \M_{x_1 ... x_{k-1}1, x_{k+1} ... x_{k+n-1} 0} \otimes [\beta]$ for $x_1x_2 ... x_{k-1}1,x_{k+1} ... x_{k+n-1} 0 \in \{0,1\}^{k+n}$ and $\alpha,\beta \in P\Sigma O_{k+n}$, with the natural left and right actions. We can easily compute the action of the generators $a_{i,j}$ on the kneading bimodules as below.

\begin{proposition}
\label{thm:comp}
For $1 \leq i, j \leq k+n$, we have the following actions of the bimodules defined by $a_{i, j} \in P\Sigma O_{k+n}$ on the bimodules $\M_{x_1x_2 ... x_{k-1} 1, x_{k+1} ... x_{k+n-1} 0}$, and by extension on the elements of $\mathfrak{G}_{k,n}$:

\begin{align*} [a_{k+1, 1}] \otimes \M_{x_1x_2 ... x_{k-1} 1, x_{k+1} ... x_{k+n-1} 0} &= \M_{x_1x_2 ... x_{k-1} 0, x_{k+1} ... x_{k+n-1} 1} \simeq \\ &\simeq \M_{\overline{x_1}\overline{x_2} ... \overline{x_{k-1}} 1, \overline{x_{k+1}} ... \overline{x_{k+n-1}} 0} \end{align*}

\begin{gather*}[a_{1, k+1}] \otimes \M_{x_1x_2 ... x_{k-1} 1, x_{k+1} ... x_{k+n-1} 0} = \hspace{3in} \\  = \M_{x_1x_2 ... x_{k-1} 1, x_{k+1} ... x_{k+n-1} 0} \otimes \left[\prod_{x_{i-1} = 0} a^{-1}_{i-1, k+n} \prod_{x_{i-1} = 1} a^{-1}_{i-1, k}\right] = \\ = \M_{x_1x_2 ... x_{k-1} 1, x_{k+1} ... x_{k+n-1} 0} \otimes \left[\prod_{x_{i-1} = 1} a_{i-1, k+n} \prod_{x_{i-1} = 0} a_{i-1, k}\right] \ \ \  \end{gather*}

For $i \neq k+1$, \begin{gather*}[a_{i, 1}] \otimes \M_{x_1x_2 ... x_{i-2}x_{i-1}x_i ... x_{k+n-1} 0} = \M_{x_1x_2 ... x_{i-2}\overline{x_{i-1}}x_i ... x_{k+n-1} 0} \end{gather*}

For $j \neq k+1$, \begin{gather*}[a_{1, j}] \otimes \M_{x_1x_2 ... x_{k-1} 1, x_{k+1} ... x_{k+n-1} 0} = \M_{x_1x_2 ... x_{k-1} 1, x_{k+1} ... x_{k+n-1} 0} \otimes \left[\prod_{x_{i-1} = \overline{x_{j-1}}} a_{i-1, j-1}\right] \end{gather*}

For $i \neq 1$, \begin{align*}[a_{i, k+1}] \otimes \M_{x_1x_2 ... x_{i-2}0x_i ... x_{k+n-1} 0} &= \M_{x_1x_2 ... x_{i-2}0x_i ... x_{k+n-1} 0} \otimes [a_{i-1, k+n}] \\ [a_{i, k+1}] \otimes \M_{x_1x_2 ... x_{i-2}1x_i ... x_{k+n-1} 0} &= \M_{x_1x_2 ... x_{i-2}1x_i ... x_{k+n-1} 0} \otimes [a_{i-1, k}]\end{align*}

For $j \neq 1$, \begin{align*}[a_{k+1, j}] \otimes \M_{x_1x_2 ... x_{j-2}0x_j ... x_{k+n-1} 0} &= \M_{x_1x_2 ... x_{j-2}0x_j ... x_{k+n-1} 0} \otimes [a_{k+n, j-1}] \\ [a_{k+1, j}] \otimes \M_{x_1x_2 ... x_{j-2}1x_j ... x_{k+n-1} 0} &= \M_{x_1x_2 ... x_{j-2}1x_j ... x_{k+n-1} 0} \otimes [a_{k, j-1}]\end{align*}

For $i, j \neq 1, k+1$, \begin{align*}[a_{i, j}] \otimes \M_{x_1x_2 ... x_{i-2}\overline{x_{j-1}}x_i ... x_{k+n-1} 0} &= \M_{x_1x_2 ... x_{i-2}\overline{x_{j-1}}x_i ... x_{k+n-1} 0} \\ [a_{i, j}] \otimes \M_{x_1x_2 ... x_{i-2}x_{j-1}x_i ... x_{k+n-1} 0} &= \M_{x_1x_2 ... x_{i-2}x_{j-1}x_i ... x_{k+n-1} 0} \otimes [a_{i-1, j-1}]\end{align*}

\end{proposition}

\begin{proof}
This follows from direct computation.

For example, the first computation is for $[a_{k+1, 1}] \otimes \M_{x_1x_2 ... x_{k-1} 1, x_{k+1} ... x_{k+n-1} 0}$. Notice that in the kneading automaton generating $\M_{x_1x_2 ... x_{k-1} 1, x_{k+1} ... x_{k+n-1} 0}$ we have $s_1 = (01), s_{k+1} = (s_{k+n}, s_k)$. So $s_{k+1}^{s_1} = (01) \cdot (s_{k+n}, s_k) \cdot (01) = (s_{k}, s_{k+n})$. So the resulting bimodule is $\M_{x_1x_2 ... x_{k-1} 0, x_{k+1} ... x_{k+n-1} 1}$, since the coordinates of the restrictions of $s_k$ have switched. As we noted above, this bimodule is isomorphic to $\M_{\overline{x_1}\overline{x_2} ... \overline{x_{k-1}} 1, \overline{x_{k+1}} ... \overline{x_{k+n-1}} 0}$.

Another example: $[a_{i, k+1}] \otimes \M_{x_1x_2 ... x_{i-2}0x_i ... x_{k+n-1} 0}$ (for $i \neq 1$). Here we have $s_{i} = (s_{i-1}, \1)$ and $ s_{k+1} = (s_{k+n}, s_k)$. So $s_i^{s_{k+1}} = (s_{k+n}^{-1}, s_k^{-1}) \cdot (s_{i-1}, \1) \cdot (s_{k+n}, s_k) = (s_{i-1}^{s_{k+n}}, \1)$. Since none of the coordinates have changed, the kneading sequence remains unaltered, but we have replaced $s_{i-1}$ with $s_{i-1}^{s_{k+n}}$, so we gain the twist $a_{i-1, k+n}$. Thus, our resulting bimodule is $\M_{x_1x_2 ... x_{i-2}0x_i ... x_{k+n-1} 0} \otimes [a_{i-1, k+n}]$.

\end{proof}

Fix a preperiodic quadratic topological polynomial $f$ with preperiod length $k \geq 1$ and period length $n \geq 2$, and recall that the $P\Sigma O_{k+n}$-bimodule $\mathfrak{G}_f$ is the set  of all $F_{k+n}$-bimodules of the form $[\alpha] \otimes \M_f \otimes [\beta]$ where $\alpha, \beta \in P\Sigma O_{k+n}$ with the natural $P\Sigma O_{k+n}$ right and left actions.

\begin{proposition}

$\G_{k,n} = \G_f$.

\end{proposition}

\begin{proof}

By Theorem \ref{thm:Nek1}, $\M_f$ is isomorphic to $\M_{x_1 ... x_{k-1}1, x_{k+1} ... x_{k+n-1} 0} \otimes [\phi]$ for some \\ $x_1x_2 ... x_{k-1}1,x_{k+1} ... x_{k+n-1} 0 \in \{0,1\}^{k+n}$ and some $\phi \in P\Sigma O_{k+n}$. By Proposition \ref{thm:comp}, the left action of $P\Sigma O_{k+n}$ is transitive on the basis of kneading bimodules. Thus, the two bimodules are equal.

\end{proof}

Now for $1 \leq i \leq n$, define $\gamma_i \in P\Sigma O_{k+n}$ by: \begin{gather*} \gamma_i = a_{k+i, k+i-1}a_{k+i, k+i-2} ... a_{k+i, k+1}a_{k+i, k+n}a_{k+i, k+n-1} ... a_{k+i, k+i+1}.\end{gather*}

So for $n=2$, we would have: \begin{gather*} \gamma_1 = a_{k+1, k+2} \\ \gamma_2 = a_{k+2, k+1}\end{gather*}

For $n=3$, we would have: \begin{gather*} \gamma_1 = a_{k+1, k+3}a_{k+1, k+2} \\ \gamma_2 = a_{k+2, k+1}a_{k+2, k+3} \\ \gamma_3 = a_{k+3, k+2}a_{k+3, k+1} \end{gather*}

For $n = 4$: \begin{gather*} \gamma_1 = a_{k+1, k+4}a_{k+1,k+3}a_{k+1,k+2} \\ \gamma_2 = a_{k+2,k+1}a_{k+2,k+4}a_{k+2,k+3} \\ \gamma_3 = a_{k+3, k+2}a_{k+3,k+1}a_{k+3,k+4} \\ \gamma_4 = a_{k+4, k+3}a_{k+4,k+2}a_{k+4,k+1}\end{gather*}

Notice that the product $\gamma = \gamma_n\gamma_{n-1} ... \gamma_2\gamma_1$ acts on $S^2 \setminus P_f$ by a Dehn Twist about a simple closed curve $\Gamma$ that separates the points $\{\omega_{k+1}, ... , \omega_{k+n}\}$ from the rest of $P_f$. That is, $\Gamma$ separates the period of the post-critical set from the preperiod and $\infty$. Since $k \geq 1$ and $n \geq 2$, the curve $\Gamma$ is non-peripheral.

\begin{lemma}\label{quadcrit}

For $g$ a pre-periodic quadratic topological polynomial with preperiod length $k \geq 1$ and period length $n \geq 2$ and $\gamma \in P\Sigma O_{k+n}$ acting by a Dehn twist $T_\Gamma$ about a simple closed curve $\Gamma$ that separates the period $\{\omega_{k+1}, ... , \omega_{k+n}\}$ from the rest of the post-critical set, if $[\gamma] \otimes \M_g \simeq \M_g \otimes [\gamma]$, then $\Gamma$ is a $g$-stable Levy cycle of length 1, and $g$ is not equivalent to any complex polynomial.

\end{lemma}

\begin{proof}

By $[\gamma] \otimes \M_g \simeq \M_g \otimes [\gamma]$ we have $T_\Gamma \circ g$ homotopic to $g \circ T_\Gamma$, so $g(\Gamma)$ must be homotopic to $\Gamma$ and $g$ must map it by degree 1. Notice also that the other component of $g^{-1}(\Gamma)$ is peripheral about the point $\omega_k$. Thus, $\Gamma$ is a $g$-stable Levy cycle of length 1 for $g$.

\end{proof}

Using the computations given in Proposition \ref{thm:comp}, we can verify the following lemma:

\begin{lemma}
\label{thm:gamma}
For any $x_1...x_{k-1} \in \{0,1\}^{k-1}$, we have $$[\gamma_1] \otimes \M_{x_1 ... x_{k-1}1, 00 ... 00} = \M_{x_1 ... x_{k-1}1, 00 ... 00} \otimes [\gamma_{n}],$$ and for $1 \leq i < n$, $$[\gamma_i] \otimes \M_{x_1 ... x_{k-1}1, 00 ... 00} = \M_{x_1 ... x_{k-1}1, 00 ... 00} \otimes [\gamma_{i-1}].$$
\end{lemma}

\begin{proof}
This follows directly from Proposition \ref{thm:comp}.
\end{proof}

We are now ready to state our first result.

\begin{theorem}
\label{thm:quadraticpreperiodic}
For $f$ a pre-periodic topological polynomial of degree 2 with preperiod length $k \geq 1$ and period length $n \geq 2$, then there exists a topological polynomial $g$ which has the same mapping scheme as $f$, but is not equivalent to any complex polynomial.
\end{theorem}

\begin{proof}
Pick a planar generating set $\{s_1, ... , s_{k+n}\}$. Let $\M_f$ be as defined earlier and let $\M_f \simeq \M_{\K(f)} \otimes [\phi]$ be the twisted kneading bimodule representation of $\M_f$ as in Theorem \ref{thm:Nek1}. We have that $\M_{\K(f)} = \M_{x_1 ... x_{k-1}1, x_{k+1} ... x_{k+n-1} 0}$ for some $x_i \in \{0,1\}$ (in fact, $x_1x_2 ... x_{k-1} 1(x_{k+1} ... x_{k+n-1} 0)$ is the kneading sequence of $f$ in the sense of \cite{BS02}).

By Proposition \ref{thm:comp}, there exists $\alpha \in P\Sigma O_{k+n}$ such that $[\alpha]\otimes \M_{\K(f)} = \M_{x_1 ... x_{k-1}1, 0 ... 0 0}$, where the $x_i$ are as in $\M_{\K(f)}$. Let $\M_0 = \M_{x_1 ... x_{k-1}1, 0 ... 0 0} = [\alpha]\otimes \M_{\K(f)}$.

Define $\beta \in P\Sigma O_{k+n}$ by $\beta = \phi^{-1}\gamma_{n-1}\gamma_{n-2} ... \gamma_1$ where the $\gamma_i$ are as in Lemma \ref{thm:gamma}. Notice that $[\alpha] \otimes \M_f \otimes [\beta] = \M_0 \otimes [\gamma_{n-1}\gamma_{n-2} ... \gamma_1]$.

Recall $\gamma = \gamma_n\gamma_{n-1} ... \gamma_1$. By Lemma \ref{thm:gamma}, we have that $[\gamma] \otimes \M_0 =  \M_0 \otimes [\gamma_{n-1}\gamma_{n-2} ... \gamma_1\gamma_n]$.

Notice that \begin{align*} [\gamma] \otimes \left( [\alpha] \otimes \M_f \otimes [\beta] \right) &= [\gamma] \otimes \M_0 \otimes [\gamma_{n-1}\gamma_{n-2} ... \gamma_1] = \\ &= \M_0 \otimes [\gamma_{n-1}\gamma_{n-2} ... \gamma_1\gamma_n] \otimes [\gamma_{n-1}\gamma_{n-2} ... \gamma_1] = \\ &= \M_0 \otimes [\gamma_{n-1}\gamma_{n-2} ... \gamma_1] \otimes [\gamma] = \\ &= \left( [\alpha] \otimes \M_f \otimes [\beta] \right) \otimes [\gamma].\end{align*}

Let $\M_g$ be the bimodule $[\alpha] \otimes \M_f \otimes [\beta] = \M_0 \otimes [\gamma_{n-1}\gamma_{n-2} ... \gamma_1]$. By Proposition \ref{thm:Nek2}, there exists a topological polynomial $g$ whose associated bimodule is $\M_g$. Notice that $P_g = P_f$ and that the mapping schemes of these two topological polynomials are the same. While the labels of the period are all 0 in the Moore diagram of $\K(g)$, they might not be in the Moore diagram of $\K(f)$.

So we have that $[\gamma] \otimes \M_g = \M_g \otimes [\gamma]$, where $\gamma$ acts on $S_2 \setminus P_g$ by the Dehn twist $T_\Gamma$ about a non-peripheral simple closed curve $\Gamma$ that separates the period of $P_g$ from the rest of $P_g$. By Lemma \ref{quadcrit}, $g$ is not equivalent to any complex polynomial.
\end{proof}

This yields the following corollary, which is also a consequence of Koch's result:

\begin{corollary}
\label{thm:quadcor}
Every quadratic mapping scheme with preperiod length $k \geq 1$ and period length $n \geq 2$ is realized by a topological polynomial that is not equivalent to any complex polynomial.
\end{corollary}

Recall that by Berstein-Levy (and an easy additional observation), every quadratic mapping scheme not meeting the hypotheses of Corollary \ref{thm:quadcor} can only be realized by topological polynomials that \emph{are} equivalent to complex polynomials.

\section{Proof of the main result}
\label{sec:arbitrary}

The proof of Theorem \ref{thm:quadraticpreperiodic} only uses the fact that the topological polynomial is quadratic to produce the basis of kneading bimodules $\M_{x_1 ... x_{k-1}1, x_{k+1} ... x_{k+n-1} 0}$. For arbitrary degree $d \geq 3$, we will not be able to produce an explicit basis in this way. However, we establish a few conditions on the mapping scheme when we can replicate the argument from Section \ref{sec:quad}.

We will show the existence of a kneading bimodule that acts as $\M_0$ for our mapping scheme. In order for the bimodule to play this role, the kneading automaton that generates it needs to have:

\begin{enumerate}

\item the states associated with the post-critical points inside the Levy cycle must restrict to each other all in the same coordinate,

\item any pair of these same states must not share any other coordinates with non-trivial restrictions, and

\item the permutations of these states must not interact with the non-trivial restrictions.

\end{enumerate}

We need (1) so that we can cycle the $a_{i,j}$ and $\gamma_i$ as in Lemma \ref{quadcrit}. We need (2) so that we do not pick up any extra generators while cycling through. Finally, we need (3) to guarantee that the presence of active states does not disturb the two previous properties.

In the case of a period of length at least two which has no critical values (in some sense, a period which is \emph{strongly} not an attractor), the proof follows very similar lines to the quadratic case. For a period containing critical values, there needs to be sufficiently many critical values outside the period so that their associated states can act on the non-trivial restriction coordinates, so that the states associated with the critical values in the period can act trivially on these coordinates. The following proposition precisely defines ``sufficiently many'' for a given period.

\begin{proposition}
\label{thm:omega}

For $S = S(C,P,\alpha, \nu )$ a polynomial mapping scheme of degree $d$ with $\Omega = \{\omega_1, ... , \omega_n\} \subset Z \setminus C$ such that $\alpha(\Omega) = \Omega$ and $$\#(\alpha^{-1}(\Omega) \setminus \Omega ) \leq \sum_{z \in C \setminus \{ \infty\}, \alpha(z) \notin \Omega} (\nu(z) - 1),$$ then there exists a topological polynomial $g$ realizing $S$ such that in the Moore diagram of the kneading automaton $\K(g)$, the states $s_\Omega = \{s_1, ... , s_n\}$ associated with $\Omega$ have all of the following properties:

\begin{enumerate}
\item the arrows from states in $s_\Omega$ to other states in $s_\Omega$ are labeled by 0,
\item the arrows from states in $s_\Omega$ to states not in $s_\Omega$ have pairwise disjoint labels from $\{1, ... , d-1\}$,
\item the sets of labels on which the states in $s_\Omega$ act non-trivially are pairwise disjoint, and
\item for all $1 \leq i, j \leq n, i \neq j$, state $s_i$ acts trivially on any letter labeling an arrow leaving $s_j$.
\end{enumerate}

\end{proposition}

\begin{proof}

Choose a planar generating set for $\pi_1(S^2 \setminus P)$. Since $S$ is a polynomial mapping scheme, there exists some topological polynomial $f$ realizing it. Let $\K(f)$ be the kneading automaton of $f$. In its Moore diagram, we will relabel some of the arrows leaving the states $s_\Omega = \{s_1, ... , s_n\}$ associated with $\Omega$ and redefine the actions of possibly all of the active states to produce the Moore diagram of a kneading automaton $\K(g)$ with the desired properties. By Proposition \ref{thm:Nek2}, we know there is some topological polynomial $g$ with this kneading automaton and no twist. So we need only describe how to produce $\K(g)$.

First, relabel all arrows within $s_\Omega$ by 0. Notice there are no arrows entering $s_\Omega$ from outside it (since $\alpha(\Omega) = \Omega$). 

Second, since $$\#(\alpha^{-1}(\Omega) \setminus \Omega ) \leq \sum_{z \in C \setminus \{ \infty\}, \alpha(z) \notin \Omega} (\nu(z) - 1),$$ we have that $$\sum_{z \notin \Omega, \alpha(z) \in \Omega}(\nu(z)-1) + \#(\alpha^{-1}(\Omega) \setminus \Omega ) \leq \sum_{z \in C \setminus \{ \infty\}} (\nu(z) - 1),$$ which by the Riemann-Hurwitz Formula gives us $$\sum_{z \notin \Omega, \alpha(z) \in \Omega}\nu(z) \leq d-1.$$ Thus, the number of arrows entering $\Omega$ in $S$ is at most $d-1$. Redefine the actions of the states in $s_\Omega$ so that the sets on which they act non-trivially are pairwise disjoint (note we can do this by the above calculation). As required by Lemma \ref{thm:cycles}, label any arrows from $s_\Omega$ to states associated with critical points by an appropriate letter on which the originating state acts non-trivially. As for the remaining arrows leaving $s_\Omega$, notice by the calculation above there are enough letters remaining in $\{1, ... , d-1\}$ so that these may be labeled from this set so that these labels are pairwise disjoint with each other and with the non-trivial actions.

Notice that states $s_\Omega$ in our new automaton now fit the requirements listed in the statement of the proposition. However, since we redefined some of the actions, we may need to redefine the actions of the states not in $s_\Omega$ in order to ensure that our automaton is kneading.

Let $k$ be the number of critical values of $S$ not in $\Omega$. Let $\mathcal{H}_0$ be the hypergraph with vertices $\{0, 1, ... , d-1\}$ and hyperedges defined by the actions of the states in $s_\Omega$. We will redefine the actions of the $k$ active states outside of $s_\Omega$ one-by-one by considering their resulting hypergraphs $\mathcal{H}_1, \mathcal{H}_2, ..., \mathcal{H}_k$. We will think of $\mathcal{H}_{i-1}$ as a sub-hypergraph of $\mathcal{H}_i$.

First, define the actions of these active states so as to connect the partial hypergraph induced by the hyperedges of $\mathcal{H}_0$, without adding cycles. By our assumption that $$\#(\alpha^{-1}(\Omega) \setminus \Omega ) \leq \sum_{z \in C \setminus \{ \infty\}, \alpha(z) \notin \Omega} (\nu(z) - 1),$$ this will eventually yield a hypergraph $\mathcal{H}_{i_0}$ where the partial hypergraph induced by the hyperedges is connected. Continue redefining the actions of the active states so as to maintain this property and not create cycles.

By a standard result in combinatorics (see e.g. Proposition 4 in Chapter 17 of \cite{Ber73}), a connected hypergraph with no cycles on $d$ vertices with $m$ hyperedges containing the vertices $\{E_i\}_{i=1}^{m}$, obeys the formula: $$\sum_{i=1}^{m}(\#E_i - 1) = d-1.$$

This, of course, is exactly the Riemann-Hurwitz Formula in our setting. Therefore, we may redefine the actions of all $k$ active states so that the final hypergraph $H_k$ is connected (notice that every vertex lies in some hyperedge) and contains no cycles. In other words, the sequence of permutations defined by our automaton is dendroid.

Let $\K(g)$ be this automaton. Notice that $\K(g)$ is dendroid by construction. Further, since its input and output sets are equal (as they were in $\K(f)$), it is kneading. As mentioned at the beginning of the proof, let $g$ be the unique topological polynomial for this planar generating set which has $S$ as its mapping scheme, $\K(g)$ as its kneading automaton, and trivial twist.

\end{proof}

\begin{lemma}\label{criterion}

For $S = S(C,P,\alpha,\nu)$ a polynomial mapping scheme of degree $d$ meeting the hypotheses of Proposition \ref{thm:omega}, $g$ a topological polynomial realizing $S$ with the same kneading automaton as the one constructed by Proposition \ref{thm:omega}, and $\gamma \in P\Sigma O_{\#P}$ acting by a Dehn twist $T_\Gamma$ about a curve $\Gamma$ separating $\Omega$ (as in the Proposition) from the rest of $P$, if $[\gamma] \otimes \M_g \simeq \M_g \otimes [\gamma]$, then $\Gamma$ is a $g$-stable Levy cycle of length 1 and $g$ is not equivalent to any complex polynomial.

\end{lemma}

\begin{proof}

As in Lemma \ref{quadcrit}, we have that $T_\Gamma \circ g$ homotopic to $g \circ T_\Gamma$, so $g(\Gamma)$ must be homotopic to $\Gamma$ and $g$ must map it by degree 1. By property (2) of the kneading automaton $\K(g)$ from Proposition \ref{thm:omega}, every component of $g^{-1}(\Gamma)$ except $\Gamma$ itself is peripheral. Hence, $\Gamma$ is a stable Levy cycle of length 1 for $g$.

\end{proof}

Our main result follows from the above proposition using the argument of the previous section.

\begin{main}

Suppose that a polynomial mapping scheme satisfies one of the following conditions:

\begin{enumerate}
\item \label{Case 1} at least one (non-attractor) period of length at least two and not containing critical values, 
\item \label{Case 2} at least two (non-attractor) periods not containing critical values,
\item \label{Case 3} at least two non-attractor periods both of length at least two, or
\item \label{Case 4} at least four non-attractor periods.
\end{enumerate}

Then this scheme is realized by a topological polynomial that is not equivalent to any complex polynomial.

\end{main}

\begin{proof}

Let $S$ be the mapping scheme satisfying one of the cases.

\begin{enumerate}

\item $S$ has one period of length at least two not containing critical values.

Let $\Omega = \{\omega _1, \omega_2, ... , \omega_n\}$ be the period given (with $\alpha(\omega_i) = \omega_{i+1}$ and $\alpha(\omega_n) = \omega_1$) and let $P = \{\infty, \omega_1, ... , \omega_m \}$. Choose a planar generating set $\{s_1, ... , s_m\}$ such that $s_i$ loops around $\omega_i$. By the Riemann-Hurwitz Formula, $C \setminus \{\infty\}$ has at most $d-1$ elements. So $\Omega$ has at most $d-1$ arrows incoming (and none outgoing) in the mapping scheme $S$. Thus, we may apply Proposition \ref{thm:omega}; let $f$ be the topological polynomial given by this proposition, and let $\M_0$ be the bimodule generated by the kneading automaton $\K(f)$.

For $1 \leq i \leq n$, define $$\gamma_i = a_{i,i-1}a_{i,i-2} ... a_{i,1}a_{i,n}a_{i,n-1} ... a_{i, i+1} $$ and $\gamma = \gamma_n\gamma_{n-1} ... \gamma_2\gamma_1$ (similar to the $\gamma_i$ in Lemma \ref{thm:gamma} but reducing all the indices by $k$).

Consider  $[\gamma_i] \otimes \M_0$. First, we have  $[a_{i,i+1}] \otimes \M_0$. Now, we do not know the full wreath recursions for the states $s_{i}$ and $s_{i+1}$, however we do know they are inactive and that $s_{i}|_0 = s_{i-1}$ and $s_{i+1}|_0 = s_{i}$. So $s_{i}^{s_{i+1}}|_0 = s_{i-1}^{s_{i}}$. Further, since none of the other arrows from these two states share labels with each other, the wreath recursions do not share non-trivial coordinates besides the one we have already considered. Thus, the only twist produced is $a_{i-1, i}$, and the kneading bimodule remains unchanged.

We may repeat the above argument for the rest of $\gamma_i$, and we find that \begin{align*} [\gamma_i] \otimes \M_0 &= [a_{i,i-1}a_{i,i-2} ... a_{i,1}a_{i,n}a_{i,n-1} ... a_{i, i+1}] \otimes \M_0 = \\ &= \M_0 \otimes [a_{i-1,i-2}a_{i-1,i-3} ... a_{i-1,n}a_{i-1,n-1}a_{i-1,n-2} ... a_{i-1, i}]  = \\ &= \M_0 \otimes [\gamma_{i-1}].\end{align*} So $[\gamma_i] \otimes \M_0 = \M_0 \otimes [\gamma_{i-1}]$ and $[\gamma_1] \otimes \M_0 = \M_0 \otimes [\gamma_{n}]$ just as in the lemma and for the exact same reasons.

Set $\M_g = \M_0 \otimes [\gamma_{n-1}\gamma_{n-2} ... \gamma_1]$. Since this is a twisted kneading bimodule, let $g$ be the topological polynomial uniquely determined by this twisted kneading bimodule and this planar generating set. Notice that $g$ has mapping scheme $S$.

As before in the proof of Theorem \ref{thm:quadraticpreperiodic}, we have that \begin{align*}[\gamma] \otimes \M_g &= [\gamma_n\gamma_{n-1} ... \gamma_1] \otimes \M_0 \otimes [\gamma_{n-1} ... \gamma_1] = \\ &= \M_0 \otimes [\gamma_{n-1} ... \gamma_1\gamma_n] \otimes [\gamma_{n-1} ... \gamma_1] = \\ &= \M_0 \otimes [\gamma_{n-1} ... \gamma_1] \otimes [\gamma] = \\ &= \M_g \otimes [\gamma].\end{align*} 

%

By Lemma \ref{criterion}, we are done.

\item $S$ has two periods not containing critical values.

By (\ref{Case 1}), we need only consider when we have $\omega_1, \omega_2 \in P$ such that $\alpha$ fixes both and neither are critical values. Let $P = \{\infty, \omega_1, \omega_2, \omega_3, ... , \omega_m\}$. Choose a planar generating set $\{s_1, ... , s_m\}$ with $s_i$ looping around $\omega_i$.

By the Riemann-Hurwitz Formula, in $S$ there are at most $d-1$ arrows incoming to $\Omega = \{\omega_1, \omega_2\}$ from outside the set (there are no outgoing arrows). So by Proposition \ref{thm:omega}, there exists topological polynomial $f$ realizing $S$ with the properties outlined in the statement of the proposition.

Let $\Gamma$ be a simple closed curve separating $\Omega$ from the rest of $P$. Note that $\Gamma$ is non-peripheral.  Let $\gamma, \gamma_1, \gamma_2$ be defined by $\gamma_1 = a_{1,2}, \gamma_2 = a_{2,1}, \gamma = \gamma_2\gamma_1$. Note that $\gamma$ acts by a Dehn twist about the curve $\Gamma$.

Let the bimodule generated by the kneading automaton $\K(f)$ be $\M_g$. Let $g$ be the unique topological polynomial determined by the (twisted) kneading bimodule $\M_g$ and the planar generating set $\{s_1, ... , s_m\}$ (this is actually the polynomial constructed in Proposition \ref{thm:omega}). Notice that $g$ has mapping scheme $S$.

Since $s_1|_0 = s_1$ and $s_2|_0 = s_2$ in $\M_g$, and $s_1$ and $s_2$ share no other non-trivial restrictions, we have that $[a_{1,2}] \otimes \M_g = \M_g \otimes [a_{1,2}]$ and $[a_{2,1}] \otimes \M_g = \M_g \otimes [a_{2,1}]$. Thus, $[\gamma] \otimes \M_g = \M_g \otimes [\gamma]$, and we may apply Lemma \ref{criterion}.

\item $S$ has two non-attractor periods both of length at least two.

Let $\Omega_1$ and $\Omega_2$ be the two periods in question. For $i = 1,2$, let $p_i = \#(\alpha^{-1}(\Omega_i) \setminus \Omega_i)$. We may assume $p_1 \leq p_2$. Notice that $$\#(\alpha^{-1}(\Omega_1) \setminus \Omega_1) = p_1 \leq p_2 \leq \sum_{z \in C \setminus \{ \infty\}, \alpha(z) \notin \Omega_1} (\nu(z) - 1).$$ So Proposition \ref{thm:omega} applies to $\Omega= \Omega_1$.

Now proceed as in (\ref{Case 1}), with the additional observation that the actions of the states associated with critical values do not affect the conjugations, since all the non-trivial actions of a single state are disjoint from the non-trivial actions and restrictions of all the other states associated with $\Omega$.

\item $S$ has four non-attractor periods.

Let $\Omega_1$ be the union of two of the non-attractor periods, and $\Omega_2$ be the union of the other two. By the same argument as above in (\ref{Case 3}), we may assume that $\Omega = \Omega_1$ satisfies the hypotheses of Proposition \ref{thm:omega}.

If $\Omega$ is the union of two periods both of length one, then proceed as in (\ref{Case 2}), with the additional observation that the actions of the states associated with critical values do not affect the conjugations, since all the non-trivial actions of a single state are disjoint from the non-trivial actions and restrictions of all the other states associated with $\Omega$.

If not, consider a period of length greater than one in $\Omega$ and proceed as in (\ref{Case 3}).

\end{enumerate}
\end{proof}

\section{Topological polynomials admitting different obstructions}
\label{sec:obstructions}

While our proofs explicitly construct a single obstructed polynomial realizing a particular mapping scheme, we can easily extend the ideas to produce many more obstructed polynomials realizing the same set of mapping schemes.

The simplest way to do so is by pre-composing with Dehn twists about the obstruction $\Gamma$ (as noted in \cite{BN06}). That is, changing the twisted kneading bimodule from $\M_g$ to $\M_g \otimes [\gamma^l]$ for any $l \in \Z$. The proof that $\Gamma$ is a Levy cycle of length 1 still holds.

See Figure \ref{fig:g} for a topological polynomial produced by our proof, Figure \ref{fig:gtwist} for this example twisted as in the previous paragraph with $l=1$, and Figure \ref{fig:guntwist} for $l=-1$.

\begin{figure}[h]
\includegraphics[width=5in]{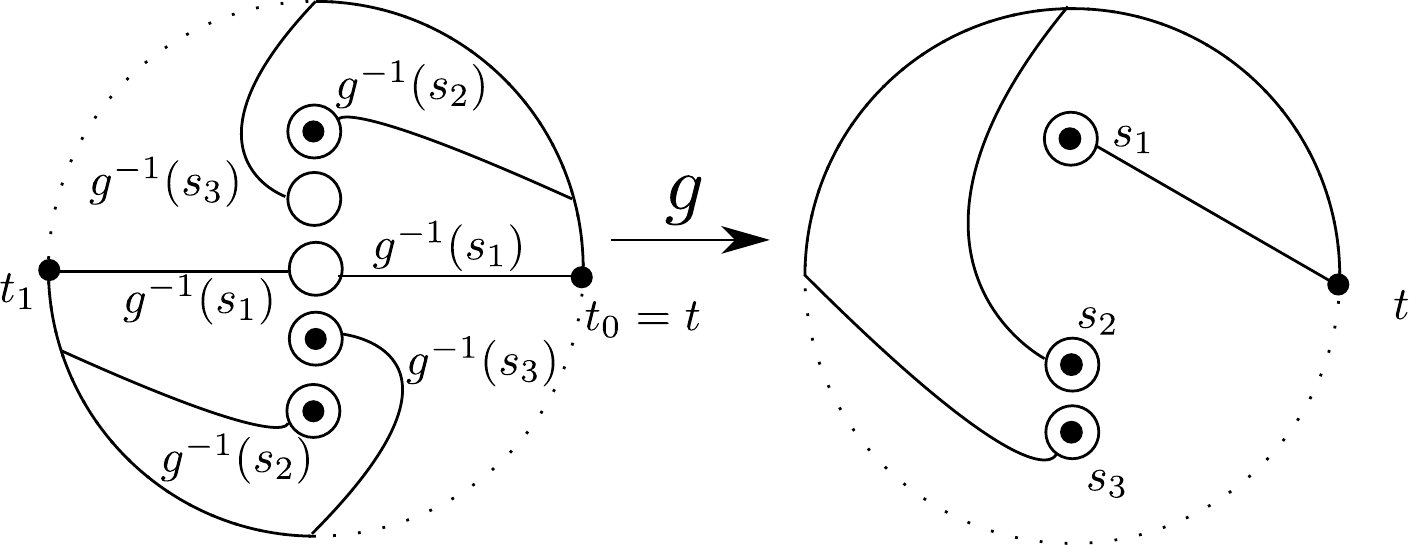}
\caption{The obstructed topological polynomial $g$}
\label{fig:g}
\end{figure}

\begin{figure}[h]
\includegraphics[width=5in]{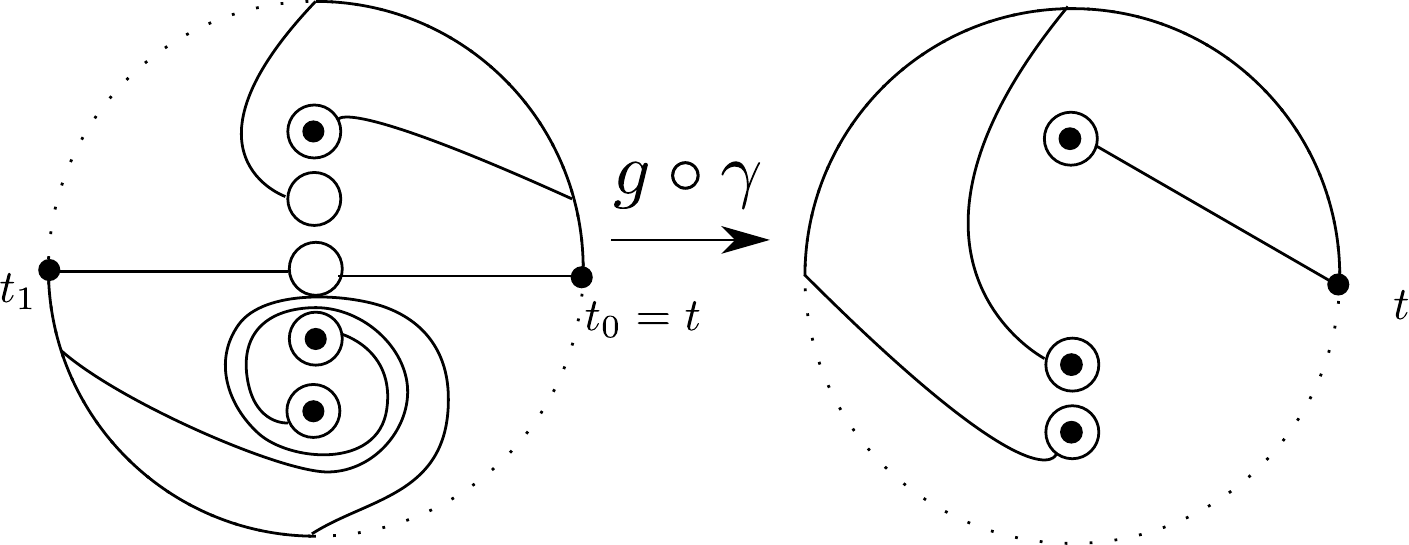}
\caption{The obstructed topological polynomial $g \circ \gamma$}
\label{fig:gtwist}
\end{figure}

\begin{figure}[h]
\includegraphics[width=5in]{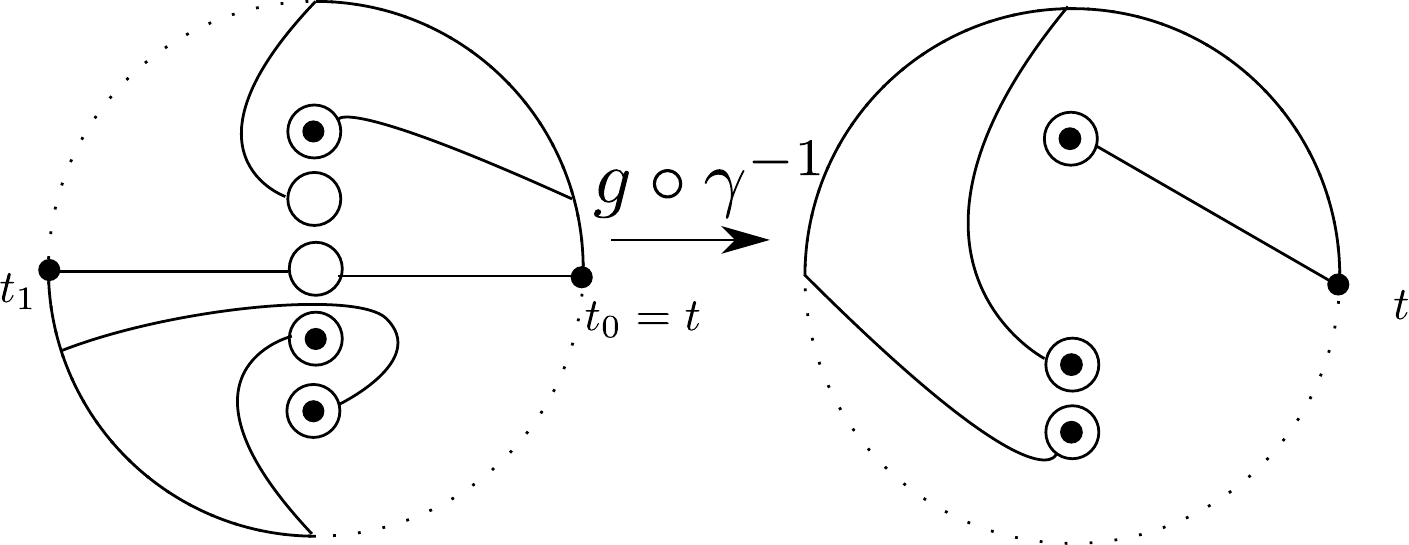}
\caption{The obstructed topological polynomial $g \circ \gamma^{-1}$}
\label{fig:guntwist}
\end{figure}

This method can also produce topological polynomials obstructed by Levy cycles of length greater than 1. 

\begin{theorem}
\label{thm:obstructions}
For $S$ a polynomial mapping scheme of degree $d$ whose post-critical set contains a period of length $n \geq 2$ which does not contain any critical values, and $1 < l < n$, $l \leq d$ such that $l$ divides $n$, then there exists a topological polynomial $g$ which realizes $S$ and admits a Levy cycle of length $l$, but does not admit a Levy cycle consisting only of a single curve surrounding the period in question.

\end{theorem}

\begin{proof}

Let $k = \frac{n}{l}$. Let $\Omega = \{ \omega_1, \omega_2, ... , \omega_n\}$ be the period given with $\alpha(\omega_i) = \omega_{i+k}$ for $1 \leq i \leq n-k$ and $\alpha(\omega_{n-k+i}) = \omega_{i+1}$ for $1 \leq i < k$, and $\alpha(\omega_n) = \omega_1$. Let $P = \{\infty, \omega_1, ... , \omega_m\}$. Choose a planar generating set $\{s_1, ... , s_m\}$ so that $s_i$ loops around $\omega_i$ in the positive direction.

Let $f$ be a topological polynomial realizing $S$. As in the proof of Proposition \ref{thm:omega}, we create a kneading automaton by copying $\K(f)$ and changing some of the labels on arrows leaving the period $\Omega$. Instead of labeling all the arrows within the period by 0, we cycle through the labels $\{0, 1, ... , l-1\}$ in the following way:

The arrows leaving $\Omega_0 = \{ \omega_1, \omega_2, ... \omega_k\}$ (but staying within $\Omega$) are labeled by 0. The arrows leaving $\Omega_1 = \{ \omega_{l+1}, ... , \omega_{2k}\}$ are labeled by 1. We continue in this way, labeling the arrows leaving $\Omega_i = \{ \omega_{ik + 1}, ... , \omega_{(i+1)k}\}$ (but staying within the period $\Omega$) by $i$ for $0 \leq i \leq l-1$.

Since there are at most $d-1$ arrows leaving the period (by the Riemann-Hurwitz Formula), there are at most $d-1$ arrows leaving the period from any one of the $\Omega_i$ defined above. Label the arrows leaving each $\Omega_i$ with pairwise distinct labels from the set $\{0, ... , d-1\} \setminus \{i\}$. Define the bimodule generated by this kneading automaton to be $\M_0$.

Define $\gamma_{i,j} \in P \Sigma O_{m}$ for $0 \leq i \leq l - 1, 1 \leq j \leq k$ by \begin{gather*} \gamma_{i ,j} = a_{ik+j,ik+j-1}a_{ik+j,ik+j-2} ... a_{ik+j,ik+1}a_{ik+j,(i+1)k}a_{ik+j,(i+1)k-1} ... a_{ik+j, ik+j+1}.\end{gather*} Let $\gamma_i = \gamma_{i,k}\gamma_{i,k-1} ... \gamma_{i,1}$, and $\gamma = \gamma_{l-1}\gamma_{l-2} ... \gamma_1\gamma_0$. 

So, for example, if $n = 4$ and $l = 2$, we would have: \begin{align*} \gamma_{0,1} &= a_{1,2} \\ \gamma_{0,2} &= a_{2,1} \\ \gamma_{1,1} &= a_{3,4} \\ \gamma_{1, 2} &= a_{4,3} \\ \gamma_0 &= a_{2,1}a_{1,2} \\ \gamma_1 &= a_{4,3}a_{3,4}.\end{align*}

Another example: if $n= 9$ and $l = 3$, we would have: \begin{align*} \gamma_{0,1} &= a_{1, 3}a_{1, 2} \\ \gamma_{0,2} &= a_{2, 1}a_{2, 3} \\ \gamma_{0,3} &= a_{3,2}a_{3, 1} \\ \gamma_{1,1} &= a_{4, 6}a_{4, 5} \\ \gamma_{1,2} &= a_{5, 4}a_{5, 6} \\ \gamma_{1,3} &= a_{6, 5}a_{6, 4} \\ \gamma_{2,1} &= a_{7, 9}a_{7, 8} \\ \gamma_{2,2} &= a_{8, 7}a_{8, 9} \\ \gamma_{2,3} &= a_{9, 8}a_{9, 7},\end{align*} which gives \begin{align*} \gamma_0 &= a_{3,2}a_{3, 1}a_{2, 1}a_{2, 3}a_{1, 3}a_{1, 2} \\ \gamma_1 &= a_{6, 5}a_{6, 4}a_{5, 4}a_{5, 6}a_{4, 6}a_{4, 5} \\ \gamma_2 &= a_{9, 8}a_{9, 7}a_{8, 7}a_{8, 9}a_{7, 9}a_{7, 8}.\end{align*}

Notice that for $i_1 \neq i_2$, we have $[\gamma_{i_1,j_1}, \gamma_{i_2, j_2}]$ trivial.

Define $\Gamma$ to be the multicurve $\{\Gamma_0, ... , \Gamma_{l-1}\}$, where $\Gamma_i$ is a nonperipheral simple closed curve separating $\Omega_i$ from the rest of $P_f$. Notice that $\gamma_i$ acts by a Dehn twist about $\Gamma_i$.

Consider $[\gamma_{i,j}] \otimes \M_0$. By construction, all the states associated with post-critical points in $\Omega_i$ pairwise share exactly one non-trivial restriction coordinate: $i$. So we have that $[\gamma_{i,j}] \otimes \M_0 = \M_0 \otimes [\gamma_{i-1, j}]$ for $1 \leq i < l-1$, $[\gamma_{0,j}] \otimes \M_0 = \M_0 \otimes [\gamma_{l-1, j-1}]$, and $[\gamma_{0,k}] \otimes \M_0 = \M_0 \otimes [\gamma_{l-1, 1}]$. Thus, $[\gamma_i] \otimes \M_0 = \M_0 \otimes [\gamma_{i-1}]$ for $1 \leq i \leq l-1$, and $$[\gamma_0] \otimes \M_0 = \M_0 \otimes [\gamma_{l-1,k-1}\gamma_{l-1, k-2} ... \gamma_{l-1, 1}\gamma_{l-1, k}].$$

Define $$\M_g =\M_0 \otimes [\gamma_{l-1,k-1}\gamma_{l-1, k-2} ... \gamma_{l-1, 1}].$$ Notice that for $1 \leq i \leq l-1$, we have $[\gamma_i] \otimes \M_g = \M_g \otimes [\gamma_{i-1}]$. Further, for $i=0$, 

\begin{align*} [\gamma_0] \otimes \M_g &= \M_0 \otimes [\gamma_{l-1,k-1}\gamma_{l-1, k-2} ... \gamma_{l-1, 1}\gamma_{l-1, k}] \otimes [\gamma_{l-1,k-1} ... \gamma_{l-1,1}] = \\ &= \M_0 \otimes [\gamma_{l-1,k-1} .. \gamma_{l-1,1}] \otimes [\gamma_{l-1}] = \\ &= \M_g \otimes [\gamma_{l-1}].\end{align*} 

For $g$ the unique topological polynomial determined by the bimodule $\M_g$ and the planar generating set $\{s_1, ... , s_m\}$, notice that $g$ has mapping scheme $S$. By the above computation, $g$ takes each curve $\Gamma_{i}$ to $\Gamma_{i+1}$ (adding mod $l$) by a degree 1 map. Further, we see by the labeling of our kneading automaton that the components of $g^{-1}(\Gamma_i)$ not labeled by $i$ are all peripheral. So the multicurve $\Gamma$ is a $g$-stable Levy cycle of length $l$.

We leave to the reader the verification that the simple closed curve surrounding the entire period is not an obstruction.

\end{proof}

For example, both the topological polynomials $g$ and $f$ in Figures \ref{fig:obstruction4} and \ref{fig:obstruction22} are quadratic with a preperiodic mapping scheme with preperiod length one and period length four. However, while $g$ admits a Levy cycle of length one (not pictured), $f$ only admits a Levy cycle of length two, shown in Figure \ref{fig:obstruction22} with its inverse images.

\begin{figure}[h]
\includegraphics[width=5in]{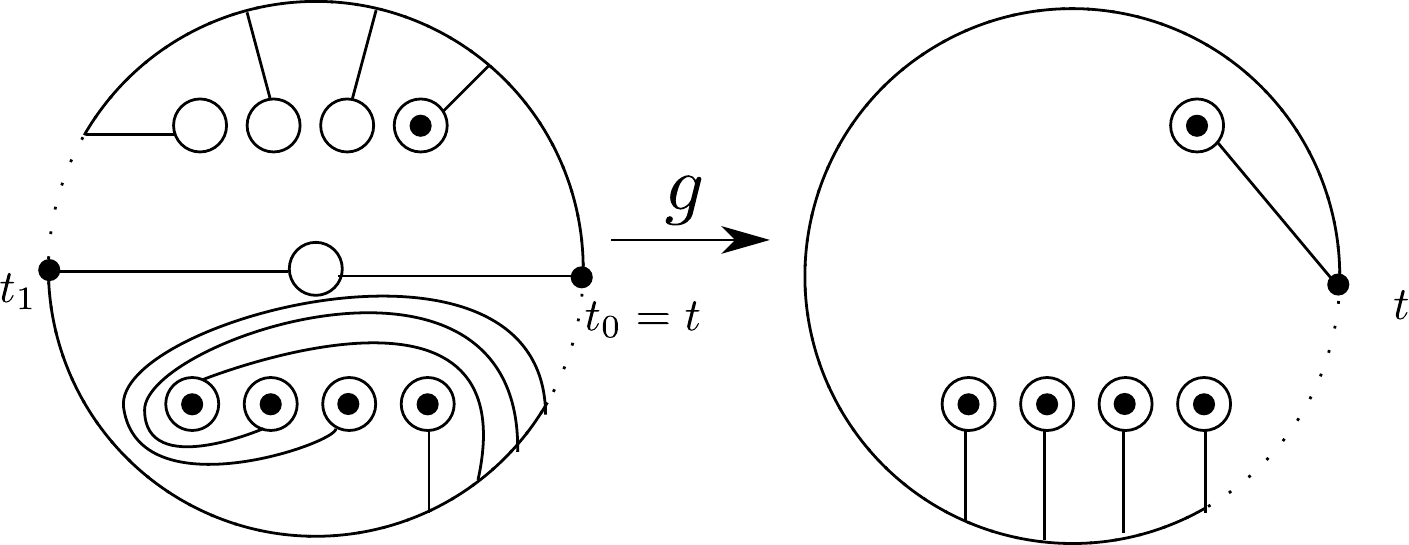}
\caption{An obstructed topological polynomial with Levy cycle of length one}
\label{fig:obstruction4}
\end{figure}

\begin{figure}[h]
\includegraphics[width=5in]{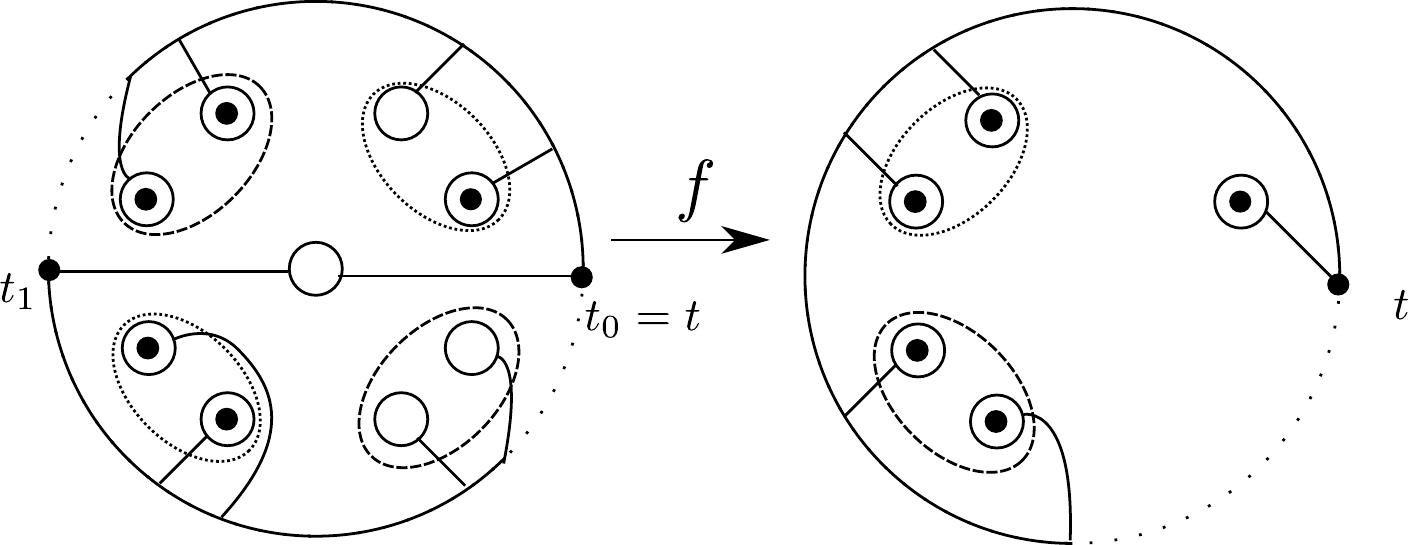}
\caption{An obstructed topological polynomial with Levy cycle of length two}
\label{fig:obstruction22}
\end{figure}

\section{Open problems}
\label{sec:future}

While there exist many polynomial mapping schemes that fall outside the purview of both Theorem \ref{thm:main} and the Berstein-Levy Theorem, in many of these cases we may still apply Proposition \ref{thm:omega} to show the existence of obstructed topological polynomials realizing the mapping scheme. For example, relatively few mapping schemes with exactly three non-attractor periods do not meet the hypotheses of Proposition \ref{thm:omega}, even if they do not meet the conditions of Theorem \ref{thm:main}.

However, there do exist non-hyperbolic mapping schemes for which Proposition \ref{thm:omega} does not apply. For instance, mapping schemes which have only a single period (say of length at least two), and that period contains all the critical values.

\begin{open}

What can we say about mapping schemes for which no $\Omega \subset Z\setminus C$ satisfies the conditions of Proposition \ref{thm:omega}?

\end{open}

Even the single period example mentioned above seems more subtle than the cases we have addressed in this paper. For instance, consider the two schemes in Figure \ref{fig:unknown_schemes}.

\begin{figure}[h]
\includegraphics[width=3in]{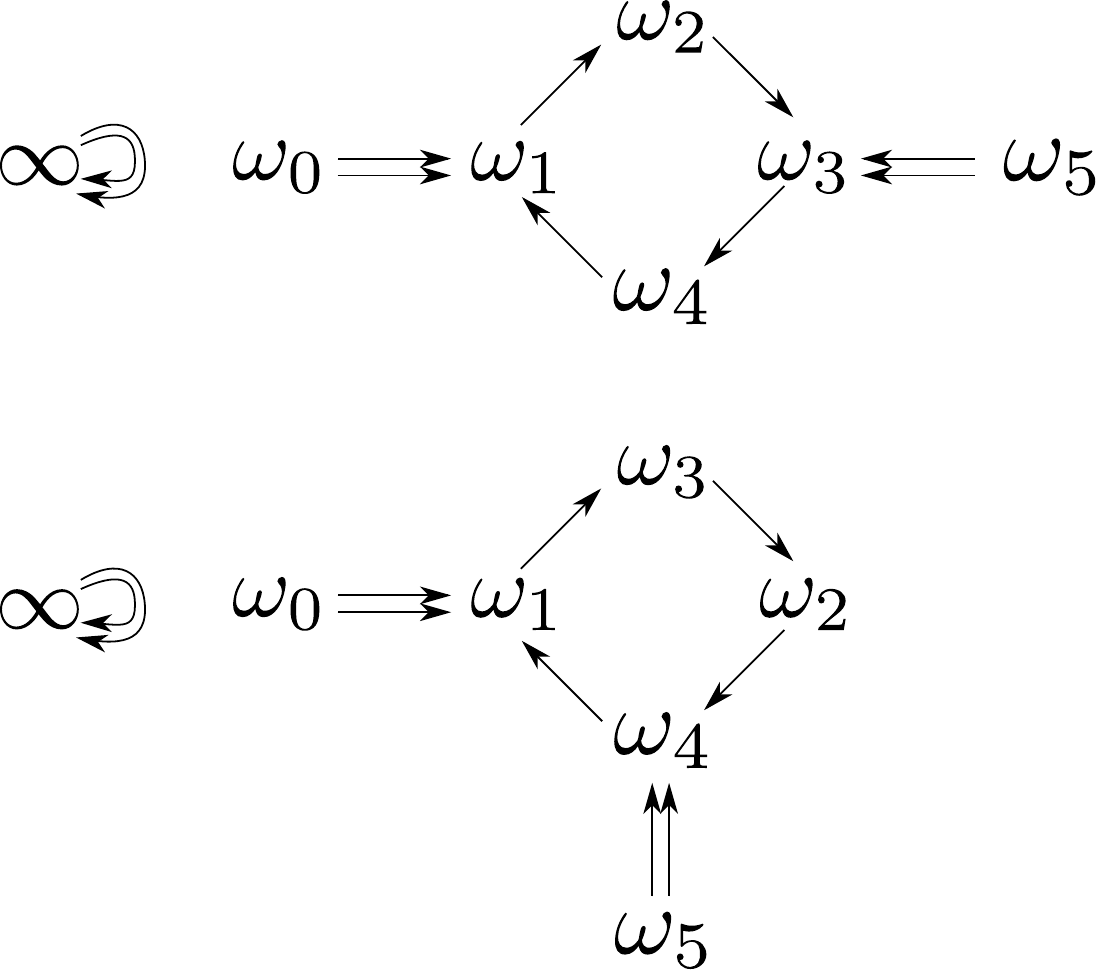}
\caption{Two mapping schemes outside the scope of our results}
\label{fig:unknown_schemes}
\end{figure}

These mapping schemes are identical in every sense that we have used to distinguish those realizable by obstructed polynomials from those not realizable. However, notice in Figure \ref{fig:unknown} that for the second scheme we may avoid the issues that arise with having all the critical values in the only period by considering a Levy cycle of length greater than 1.

\begin{figure}[h]
\includegraphics[width=5in]{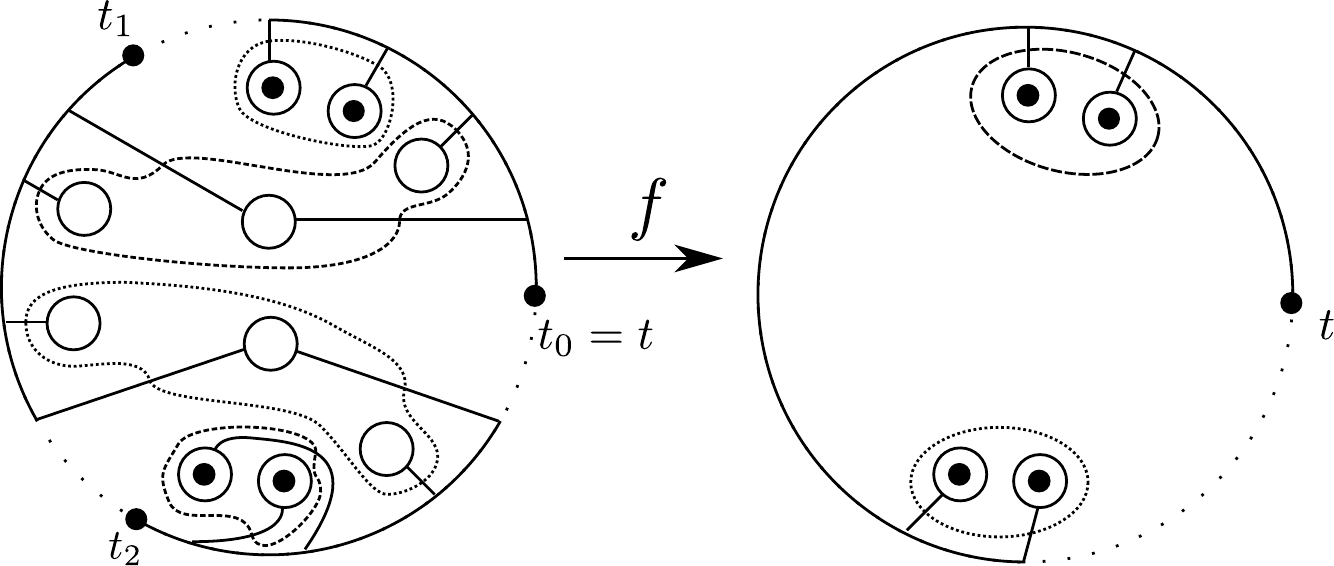}
\caption{An obstructed topological polynomial outside the scope of our results}
\label{fig:unknown}
\end{figure}

Somehow, the relative position of the critical values within this period affects the mapping scheme's realizability by obstructed topological polynomials. It seems that one could combine the ideas of Theorem \ref{thm:obstructions} with the treatment of critical values in Proposition \ref{thm:omega} to gain some ground within this class of mapping schemes.

\begin{open}

How does the relative position of critical values within a non-attractor period of a mapping scheme affect the scheme's realizability by obstructed topological polynomials?

\end{open}

However, it remains unclear how to prove the (probable) non-realizability of the remaining mapping schemes.

\begin{open}

Can we employ the $P\Sigma O_n$-bimodule theory of Nekrashevych to prove the non-realizability of mapping schemes by obstructed topological polynomials?

\end{open}

We have only produced stable Levy cycles for polynomials. Perhaps we could use the machinery of this paper to find non-stable Levy cycles for some of these unclassified schemes. Instead of showing that the bimodule $[\gamma]$ restricts to itself under $\M_g$, we would need to show that a particular product of powers of Dehn twists restricts to a different particular multitwist (these extra powers and twists coming from the other non-peripheral preimages of the Levy cycle). While the bimodules certainly still make these computations straightforward, one would need a systematic way to determine which multitwists to consider.

\begin{open}

Can we extend these results if we consider non-stable Levy cycles?

\end{open}

Also, while we have demonstrated that certain mapping schemes are realizable by obstructed topological polynomials, we have not determined if the examples produced constitute \emph{all} obstructed topological polynomials with these mapping schemes, or if there exist others with different forms. It would be interesting to know which is true.

\begin{open}

Do there exist obstructed topological polynomials realizing these mapping schemes other than those given in this paper?

\end{open}

Further, we have not attempted to determine when these various examples of obstructed topological polynomials are equivalent. A result like that for the quadratic polynomials with preperiod length 1 and period length 2 in \cite{BN06} would be interesting, even if only for preperiodic quadratic polynomials.

\begin{open}

What are the Thurston equivalence classes of these obstructed topological polynomials?

\end{open}

Also related to our results is Nekrashevych's conjecture that the bimodule $\G_f$ is sub-hyperbolic. \cite{Nek09} If this conjecture is true, then the various elements represented by $\gamma$ in our proofs have finite order in the faithful quotient of the action.

\begin{open}

What are the orders of the obstructing elements in the faithful quotient of $P\Sigma O_n$ by the self-similar action?

\end{open}

\end{document}